\documentclass[review]{elsarticle}

\usepackage{amsmath,amsthm,amssymb}
\usepackage{array}  
\usepackage{algorithm}
\usepackage{algpseudocode}
\usepackage{booktabs}

\usepackage{caption}
\usepackage{subcaption}
\usepackage{graphicx} 
\graphicspath{{./figures/}}
\usepackage{makecell}
\usepackage{multirow}

\usepackage{xcolor}
\usepackage{lineno}
\usepackage[numbers]{natbib}

\usepackage{amsmath,amsthm}
\usepackage{amssymb}
\usepackage{amsfonts}
\usepackage{bm}
\usepackage[bookmarks=true, hidelinks]{hyperref}
\hypersetup{pdfauthor={Name}}

\usepackage{caption}

\usepackage{tikz}
\usetikzlibrary{calc,fadings,decorations.pathreplacing}

\usepackage[skins,theorems,many]{tcolorbox}
\tcbset{highlight math style={enhanced,
colframe=red,colback=white,arc=0pt,boxrule=1pt}}

\usepackage[english]{babel}
\usepackage{blindtext,tikz}
\usetikzlibrary{calc}
\usepackage{fancyhdr}
\usepackage{indentfirst}
\usepackage{color}

\usepackage{listings}

\usepackage{url}
\usepackage[numbers]{natbib}

\theoremstyle{definition}
 \newtheorem{thm}{Theorem}[section]

\newtheorem{rem}[thm]{Remark}
\newtheorem{exm}[thm]{Example}

\newcommand{\Real}{\mathbb{R}}

\newcommand{\abs}[1]{\left\vert#1\right\vert}

\newcommand{\set}[1]{\left\{ #1\right\}}

\newcommand{\cK}{\mathcal{K}}

\newcommand{\rone}[1]
{{\color{black} #1}}

\newcommand{\zz}[1]{{\color{black} #1}}
\newcommand{\tw}[1]
{{\color{black} #1}}
\usepackage[letterpaper,margin=1.0in,bottom=1.1in]{geometry}

 \journal{}

 \allowdisplaybreaks[2]
\begin{document}
\begin{frontmatter}

\title{Tensor neural networks for high-dimensional Fokker-Planck equations}

\address[zz]{Department of Mathematical Sciences, Worcester Polytechnic Institute, Worcester, MA, USA}

\address[zh]{Department of Computer Science, National University of Singapore, Singapore, 119077}
\address[gk]{Division of Applied Mathematics, Brown University, Providence, RI 02912, USA} 
\address[gk1]{Advanced Computing, Mathematics and Data Division, Pacific Northwest National Laboratory, Richland, WA, United States}
\author[zz]{Taorui Wang} \ead{twang13@wpi.edu}
\author[zh]{Zheyuan Hu}   \ead{e0792494@u.nus.edu}
\author[zh]{Kenji Kawaguchi} \ead{kenji@nus.edu.sg}
\author[zz]{Zhongqiang Zhang\corref{cor1}} \ead{zzhang7@wpi.edu}
\author[gk,gk1]{George Em Karniadakis}\ead{george\_karniadakis@brown.edu}


\date{\today}


\begin{abstract}
   We solve high-dimensional steady-state Fokker-Planck equations on the whole space by applying tensor neural networks. The tensor networks are \zz{a linear combination of} tensor product\zz{s} of one-dimensional feedforward networks or a linear combination of several selected radial basis functions.  
   The use of tensor feedforward networks allows us to efficiently exploit
   auto-differentiation \zz{(in physical variables)} in major Python packages while using radial basis functions can fully avoid auto-differentiation, which is rather expensive in high dimensions. 
   We then use the physics-informed neural networks and stochastic gradient descent methods to learn the tensor networks. One essential step is to determine a proper bounded domain or numerical support for the Fokker-Planck equation.
   To better train the tensor radial basis function networks, we impose some constraints on parameters, which lead to relatively high accuracy. 
  We demonstrate numerically that the tensor neural networks in physics-informed machine learning are efficient for steady-state  Fokker-Planck equations from two to ten dimensions.  
\end{abstract}

\end{frontmatter}

\section{Introduction}
Fokker-Planck equations are often used to describe probability density functions for dynamical systems perturbed by random noise, such as in statistical mechanics, biology, finance, and engineering. 
Solutions to Fokker-Planck equations are often obtained numerically; see e.g. in  \cite{Risken96-book}.  
However, numerical methods for 
 Fokker-Planck equations often require to address the following issues: 
  approximation in unbounded domains and high dimensions.
  Thus, solving the Fokker-Planck equation can be challenging with grid-based numerical methods such as finite difference or finite element methods, especially in high dimensions. 

Instead of grid-based methods,  mesh-free numerical methods via neural networks such as physics-informed neural networks \cite{raissi2019physics} may be applied to alleviate these difficulties in high dimensions.
The main idea is to formulate the problem of solving partial differential equations (PDEs) as solving optimization problems
 while using neural networks as approximate solutions.
 Some popular choices of \zz{formulations} are using the L2-norm of the residuals of 
 the underlying PDEs (known as physics-informed neural networks, PINNs \cite{raissi2019physics}) or using the variational formulation, e.g., \cite{JordanKD1998variationalFP} and \cite{Liu22NeuralParametricFP}.  
%
  The derivatives in the loss are usually computed via auto-differentiation and back-propagation for deep neural networks. The loss functions are realized by sampling random collocation/training points. 
  This leads to discrete loss functions for optimization over the parameters of deep neural networks, which are often trained/optimized with stochastic gradient methods such as ADAM methods \cite{kingma2014adam}. This approach has been successfully applied to many problems; see e.g., in the reviews \cite{EHanJ22review,karniadakis2021physics}  and many related works. 

The challenges of machine learning algorithms for high-dimension PDEs lie in many aspects, such as 
\textit{
1) effective loss functions,  
2) computing derivatives inexpensively, 3) effective training points, and 4) efficient training methods.} 
These issues are somewhat addressed in the literature, see Section \ref{sec:issue-sml}. 
However, some important issues have not yet been explored. 
For example, for high dimensional PDEs, solutions may be small in most regions where solutions are still meaningful.  Resolving such small but meaningful scales can be challenging. A simple example is the computation of the density function
$(2\pi)^{-\frac{d}{2}}\exp (-\frac{\abs{x}^2}{2})$. When $\abs{x}=1/2$ and $d=10$, the function values are about $10^{-4}$. 
Another issue in high dimensional problems is the lack of effective benchmark problems/solutions to test. For example, 
in standard tests for high dimensional integration  \cite{genz84}, the functions are usually tensor-product functions or composition of simple functions with a weighted summation of inputs or quadratic functions of the input arguments. To the best of our knowledge, the target
solutions/functions considered in solving high-dimensional PDEs via networks are simple and most are radial basis functions. These benchmark problems effectively test the algorithms but do not address the approximation issues in high dimensions.

There have been a few works on machine learning algorithms for Fokker-Planck equations, see Section \ref{sec:fokker-planck-nn}. 
However, in addition to the issues above, the following challenges have not been addressed for Fokker-Planck equations, including 
\textit{a)
finding effective numerical supports or picking the computational domains;
b) the design of the networks for large approximation capacity \rone{ with efficiency}; and  
c) numerical integration in high dimensions crucial for numerical probability density function, whose integral must be close to 1.} 

\subsection{\rone{Our proposed methodology and contributions}}
In this work, we employ the framework of
physics-informed neural networks for Fokker-Planck equations as in 
\cite{alhussein2023physics,Dobson22-fokkerplanck,zhai2022DL4FP}
using tensor neural networks. However, 
the loss function herein is different from these works. Specifically, we use the loss function \eqref{eqn: Loss}, where the integral of the solution being 1 is implemented via normalization instead of penalization. 
Compared to \cite{Dobson22-fokkerplanck,zhai2022DL4FP}, where a large amount of data is generated from the Monte-Carlo simulations of stochastic differential equations (SDEs), we use only little data before the training.  
Specifically, we select numerical supports from a few simulated trajectories of corresponding SDEs and find this selection addresses the issue a) above in 6 dimensions. 
This step is crucial as numerical supports determine the computational domain and highly
affect the number of training points and thus the computational cost. 

To address the issue b), we use the tensor neural networks: tensor radial basis function networks and tensor feedforward networks for steady-state Fokker-Planck equations 
in high dimensions (from 2 to 10 dimensions). 
Both networks are universal \zz{approximators}, i.e., they can approximate continuous functions on \zz{bounded domains} arbitrarily well; see \ref{app:universal}.
The tensor feedforward neural networks have been proposed for low-dimensional problems, e.g., in \cite{cho2022separable,wang2022tensor, wang2022solving,zhang2023low-rank} but have not been tested in high dimensions. 
Also, tensor radial basis function networks 
have not been investigated for high dimensional PDEs, especially for Fokker-Planck equations.  

The use of tensor radial basis function networks in 
high dimensions can be advantageous. 
We can avoid auto-differentiation in computing the derivatives of each term in loss functions, \zz{ in both physical variables and parameters of radial basis function networks}. In fact, due to the structure of tensor radial basis function networks,
the derivatives in the spatial variables and 
the parameters of the networks are simple to compute and thus save huge storage required by the prevalent
auto-differentiation Python packages.  
Compared to tensor feedforward neural networks, training tensor radial basis function networks is more demanding. 
We propose also an effective initialization and impose constraints for different {parameters} for the tensor radial basis function networks 
and provide a pragmatic approach for efficient approximations in high dimensions.  Furthermore, we compare these two types of networks in Section \ref{sec:numerical}.

The issue c) can be well accommodated by 
tensor radial basis function networks as it is straightforward to compute the integration. 
For tensor feedforward neural networks, we use accurate and reliable quadrature rules for numerical integration - Gauss quadrature is accurate for smooth integrands while Monte Carlo may be universal but less accurate.

In addition, we address the issue of approximation of complicated functions in high dimensions, which has not been investigated thoroughly in the literature.  While the target solutions/functions in high dimensions are simple, e.g. linear or radial basis functions, our solutions are not Gaussian or Gaussian mixture, and they have complicated interactions among different coordinates and multiple modes. 
Here we assume that solutions are negligible outside of the numerical supports but involve non-trivial interactions among directions/arguments/dimensions. 
We observe that in high dimensions,
the training results are sensitive to the numerical support; see Example \ref{exm:10d-multi-mode} where smaller numerical supports \zz{may be needed for efficiency and accuracy} in ten dimensions. 
%
We also introduce a metric for measuring the approximation errors in high dimensions while focusing on high-probability regions. 

\rone{{We summarize below our contributions and novelty of this work. 
\begin{itemize}
\item 
We develop an algorithm to estimate a proper bounded domain (numerical support) for efficient computation. To our knowledge, these domains are often determined arbitrarily in literature for Fokker-Planck equations. 
When computational bounded domains are not used, a large amount of data from stochastic differential equations is used, e.g. in \cite{Dobson22-fokkerplanck,zhai2022DL4FP}. In this work, we only need a small amount of data to determine the support, and the determination procedure is offline while no data is used in the later training process.  This aspect distinguishes this work from other works in the literature mentioned above and in the next section.

\item  Density solutions for Fokker-Planck equations are guaranteed using tensor radial basis neural networks. The normalization constant for the density function is simple to compute due to the tensor-product in the tensor neural networks. 
For tensor feedforward neural networks, we use high-order integration or analytical methods in each dimension and the network solutions are very close to density solutions - the integration is sufficiently close to 1.
In contrast, feedforward neural networks in high dimensions are accompanied by the Monte Carlo methods or quasi-Monte Carlo methods when the high-dimensional integral is computed.
Thus the network solution may become less inaccurate.  Our numerical examples show that the inaccuracy may lead to large errors or even failure in high dimensions.  

\item  We demonstrate through several examples that tensor neural networks are efficient for high-dimensional (two- to ten-dimensional) Fokker-Planck equations.  Extensive experiments in several dimensions and ablation studies are presented to show the efficiency of the proposed methodology.  
\end{itemize}
}
} 

This paper is organized as follows.  In Section \ref{sec:literature}, we summarize works on high-dimensional PDEs and those focusing on Fokker-Planck equations. 
We present in Section \ref{sec:methods-set-up} the problem of interest and the methodology in this work.
In Section \ref{sec:numerical}, we present five examples in several dimensions. We also compare tensor radial basis function networks and tensor feedforward networks in the first three examples. 
\rone{A comparison study and ablation studies are presented in Section \ref{sec:ablation}.}
We summarize our work in Section \ref{sec:summary} and discuss some limitations of the work and possible directions to be addressed in the future.

%
\section{Literature review}\label{sec:literature}
In this section, we review first some general issues and related literature in machine learning methods for high dimensional PDEs and then discuss specific problems in methods for Fokker-Planck equations. 
Here we do not intend to have a complete list of related works as it is almost impossible due to the fast development in literature. Instead, we focus on some fundamental issues of interest to our work. 
\subsection{General issues in high dimensions}
\label{sec:issue-sml}
Some challenges stated above for high dimensional problems have been  addressed in the literature.  

For effective loss function, we may use 
the combination of the learning approach and the probabilistic representation  (e.g., Feynman-Kac) of  the solutions, i.e., representation   using solutions  stochastic differential equations, see e.g.     DeepBSDE  \cite{han2018solving, han2017deep} and  many works along this direction  
\cite{beck2019machine,beck2020overcoming_ac,beck2021deep,becker2021solving,chan2019machine,henry2017deep,hure2020deep,hutzenthaler2020overcoming,ji2020three,raissi2018forward,zhang2022predictorcorrector} and \cite{EHanJ22review} for a review. 
We remark that this approach addresses only local approximations such as solutions at one or a few points and may require extensive computations for solutions in a large region.

 For computing derivatives inexpensively, several approaches have been applied for feedforward neural networks: 
\begin{itemize}
    \item a)  Randomized calculus for  \cite{he2023learning}, by replacing the approximate solution and its derivatives with their stochastic approximation by Stein's identity.
    The replacement leads to no auto-differentiation and no back-propagation for derivatives in physical variables.
   This methodology may be used in higher dimensions but is limited to no larger than $4$ dimensions  \cite{he2023learning} \zz{ when solutions are nontrivial}. {\color{black} Randomized calculus uses approximation for differentiation, thus introducing extra approximation and requiring extra memory. 
      }
    \item b) 
  Tensor neural networks for low dimensional partial differential equations. Wang et al. \cite{wang2022tensor, wang2022solving} propose tensor neural networks for solving  Schr\"odinger equations. \cite{cho2022separable} applied the so-called separable PINNs for PDEs. In \cite{zhang2023low-rank},  a low-rank neural network is applied to solve PDEs. The networks therein are of the form 
$\sum_{l=1}^r \bigotimes_{j=1}^d NN_j^{l}$ where $NN_j^{l}$'s are neural networks of one dimension.  \textcolor{black}{Here $r$ is the rank, $l$ is the index of rank, $d$ is the input dimension and $\bigotimes$  means tensor product}. In these papers, problems in up to four dimensions are considered.
We will use tensor neural networks for higher dimensional problems.
\item 
c) A random batch method in dimensionality has been developed in \cite{hu2023tackling}. Therein, the computational cost of derivatives of feed-forward neural networks consists of
computing the derivative in a few randomly picked dimensions and repeating the process. 
\end{itemize}

As an alternative to feed-forward neural networks, radial basis function or more generally kernel function networks such as additive kernels or mixtures of kernels have been applied to accommodate the complex structures in the underlying problems and thus increase the capacity of the kernel approximations.    
For example,  additive kernels are used to solve low-dimensional PDEs, e.g. in \cite{mishra2018hybrid, MisFS19,SenvA22}.
They are also widely used in machine learning algorithms, e.g., \cite{baek_new_2019}
in support vector machines. 
An extension of additive kernels is known as multiple kernels in multiple kernel learning where a linear or non-linear combination of base kernels is used.  Many algorithms of multiple kernel learning for large-scale problems have been proposed to reduce computational complexity and memory usage, e.g.,  in \cite{AIOLLI2015215,jain2012spf,kloft2011lp,moeller2014geometric,orabona2011ultra,JMLR:v9:rakotomamonjy08a,sonnenburg2006large,varma2009more}.  
It has been shown recently that 
additive kernels can learn complex dynamics, e.g. \cite{OwhYoo19,BouOwh21,ChenOS21,YooOwh20}.
{\color{black} In Section
\ref{sec:ablation}, we will compare the performance of one popular optimizer developed for multiple kernel learning \cite{JMLR:v12:gonen11a} and stochastic gradient descent methods, e.g., ADAM \cite{kingma2014adam} and LION \cite{chen2024symbolic}.}

\subsection{Solving Fokker-Planck equations with neural networks} \label{sec:fokker-planck-nn}

The methods above may be applied to general PDEs and are not particularly designed for Fokker-Planck equations. Below we present a compact review of some recent works on 
numerical Fokker-Planck equations via neural networks. 

The formulation of physics-informed neural networks for Fokker-Planck has been used, e.g.,  in \cite{alhussein2023physics,Dobson22-fokkerplanck,zhai2022DL4FP}. 
Choices of loss functions other than the least-square formulation in physics-informed neural networks are also made. 
In 
\cite{Liu22NeuralParametricFP}, the loss function  is not the $L2$-norm of residual of PDEs but a variational formulation as in \cite{JordanKD1998variationalFP}. 
Also, 
the high dimension equation is transformed into a system of finite dimensional ordinary differential equations in \cite{Liu22NeuralParametricFP}.

In addition to feedforward neural networks, generative models are also applied for PDEs with density solutions, especially for approximating density functions.
See e.g., \cite{boffi2023probability,Liu22NeuralParametricFP,tang2022adaptive}
for the use of flow-based networks, 
\cite{anderson2023fisher,Tabandeh22FPmixtures} for Gaussian mixtures. 

If observation data is available, data can be incorporated into loss functions. For example,  the authors in  \cite{Dobson22-fokkerplanck,zhai2022DL4FP} use sampled solutions of the corresponding SDEs by Monte Carlo simulation as training points and also in data loss. For time-dependent Fokker-Planck equations, it is remarked in \cite{Li2023time-dependentFP} that concentrating the training at the initial and the terminal time can improve the training efficiency. 
In  \cite{Lin2022invariant-distribution,Lin2023invariant-distribution}, the goal is to compute generalized potential via given data, i.e  
the logarithm of the invariant distribution for time-independent Fokker-Planck. The logarithmic transformation of the solution is also applied in  \cite{mandal2023learning}  for
time-dependent Fokker-Planck equations. 
In \cite{mandal2023learning}, the long short-term memory model is used with physics-informed neural networks to solve the time-dependent Fokker-Planck equations. 
With a large amount of data generated by an SDE,  \cite{GuHLY23stationary} aims to find accurate density functions by first estimating the drift and diffusion coefficients of the SDE and then solving the
Fokker-Planck equation with estimated coefficients. 


Radial Basis Function (RBF) methods have been utilized to solve PDEs, as noted in \cite{870037}. 
Techniques such as Kansa's approach, Hermite interpolation, and the Galerkin method have been utilized to implement these RBFs in solving the Fokker-Planck Equation, as detailed in \cite{KAZEM2012181} and \cite{DEHGHAN201438}. Furthermore, \cite{WANG2023103408} employs the RBF neural network, specifically using the Gaussian RBF function and normalization, to solve the steady-state Fokker-Planck equation. However, these works address only problems {\em in low dimensions}.

\section{\zz{Problem and methodology}}\label{sec:methods-set-up}

Consider the following Fokker-Planck equation
\begin{equation}\label{eq:fokker-planck-steady}
\mathcal{L}p = - \sum_{i = 1}^d\frac{\partial }{\partial x_i}\left( f_i p \right)(\boldsymbol{x}) + \frac{1}{2} \sum_{i, j} \frac{\partial^2 }{\partial x_i \partial x_j} \left( D_{ij} p \right)(\boldsymbol{x})=0,\quad \boldsymbol{x}= (x_{1},\ldots,x_{d}) \in  \Real^d,
\end{equation}
with $\displaystyle \int_{\Real^d} p(x) dx = 1, p\geq 0 $.
We are limited to a class of Fokker-Planck equations, whose solution serves as the unique invariant distribution of   \zz{ the stochastic differential equation (SDE) \eqref{eqn:fokker-planck-sde}. Let $W_t$ be a $n$-dimensional Brownian motion and $\boldsymbol{\sigma}\in \Real^{d\times n}$. The corresponding SDE is formulated as follows}. 
\begin{equation}\label{eqn:fokker-planck-sde}
 d \boldsymbol{X_t} = \boldsymbol{f}(\boldsymbol{X_t}) dt + \boldsymbol{\sigma}(\boldsymbol{X_t}) d \boldsymbol{W_t},  \, \boldsymbol{X}_t\in \Real^d, \quad \boldsymbol{D} = \boldsymbol{\sigma} \boldsymbol{\sigma}^\top \in\Real^{d\times d}.
\end{equation}  
We refer interested readers to \cite{HuangJLY15-existence-steadyFP} for theoretical considerations of steady-state Fokker-Planck equations. 
\zz{The property of probability density function suggests that \zz{ the solution $p$ decays $0$ as $x$ goes to infinity in the Euclidean norm}. In this paper, we assume that}
\begin{enumerate}

\item \zz{The solution of equation \eqref{eq:fokker-planck-steady} is concentrated in a bounded domain $\Omega$.}
\item \zz{Outside the bounded domain $\Omega$, the solutions are almost zero. Along the boundary of $\Omega$, $p$ rapidly decay to 0.}

\end{enumerate}
\zz{
These assumptions allow us to find an open bounded domain for efficient simulations and introduce the} \rone{boundary condition $p(x) = 0
$ for $x$ belongs to $\partial \Omega$, the boundary of bounded domain $\Omega$.} Instead of arbitrarily assuming large domains, we estimate the computational domain by using  
 some simulated trajectories of the solution to the SDE \eqref{eqn:fokker-planck-sde}, see Section \ref{sec:bounded domain}.
Once a satisfactory domain is found, we use physics-informed neural networks, where the network approximates a density function. 
\rone{We illustrate our methodology for solving Equation \eqref{eq:fokker-planck-steady} in Figure \ref{fig:flow-chart-main}.}
\begin{rem}
Other than some data used in finding the computational domain, 
we don't use experimental data or simulation data  from SDE \eqref{eqn:fokker-planck-sde}.
Thus, our method is not data-driven, compared with methods  in 
\cite{Dobson22-fokkerplanck,zhai2022DL4FP}. 
\end{rem}

\begin{figure}[ht]
\centering
 \includegraphics[scale=0.8]{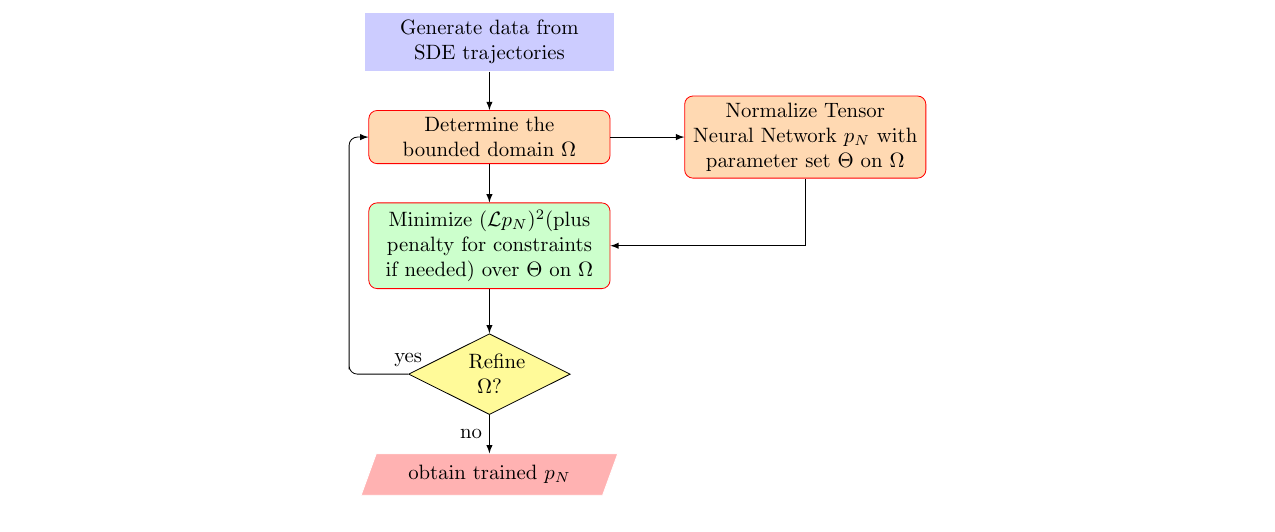}
 \caption{Flowchart of the proposed methodology for solving the Fokker-Planck equation}
 \label{fig:flow-chart-main}
\end{figure}
 

\subsection{Estimating Bounded Domain}\label{sec:bounded domain}
To improve the efficiency of computation, we introduce a bounded domain $\Omega \subset R^d$, on which the solution $p^*$ to Equation \eqref{eq:fokker-planck-steady} is highly concentrated such that  $1-\int_\Omega p^* \,dx$ \zz{ is nonnegative and  sufficiently small}. We also require  $p^*$ to be extremely close to zero outside the domain $\Omega$. 
\zz{To avoid confusion, we will use the terms bounded domain and numerical support interchangeably in this paper, both referring to the bounded domain $\Omega$.}

For $\Omega$, we use the following hypercube
$
\Omega = \bigotimes_{j=1}^d [O_j - r, O_j + r],
$ 
where $O_j$ and $r$ are real numbers. 
The hypercube facilitates the integration of our tensor neural networks over the domain, see 
Section \ref{sec:TNN} for details. 
\rone{We estimate the domain $\Omega$ through SDE trajectories. Specifically, to determine $O_j$, we apply the  Euler-Maruyama scheme of the SDE \eqref{eqn:fokker-planck-sde} \zz{up to a sufficiently terminal time}. Similar to the procedure in \cite{zhai2022DL4FP}, we choose a sufficiently large burn-in time, smaller than the terminal time, to ensure the trajectory points between the burn-in time and the terminal time represent the invariant distribution of the \eqref{eqn:fokker-planck-sde} (See \cite{levin2017markov}, Chapter 4). 
We denote the set of trajectory points between the burn-in time and the terminal time by $\mathbf{\Xi} \subset R^d$. }
\zz{We choose $\boldsymbol{O} = (O_1, O_2,..., O_d) 
=\frac{\sum_{\boldsymbol{z_i} \in \mathbf{\Xi}} \boldsymbol{z_i}}{\# \mathbf{\Xi}}$.}
To estimate $r$ for the hypercube $\Omega$, we first find a  constant $B$ such that 
$ 
B =  C \max_{j \in \{1,...d\}}( \max_{\boldsymbol{z_i} \in \mathbf{\Xi}}(O_j - z_{ij}) ), 
$ 
where we choose $C  > 1$ so that $\Omega$ includes all the samples from the simulation. 
\zz{We set $r=B$ and call $B$ the initial half-edge length. This process can be summarized in the Algorithm \ref{alg:obtain-bounded-domain}:}

\begin{algorithm}
\caption{Estimate the
numerical support $\Omega$}
\label{alg:obtain-bounded-domain}
\begin{algorithmic}
\State{\textbf{Input:}} step size $h$, burn-in step $t_{burnin}$, terminal step $t_{terminal}$, functions $\boldsymbol{f}$ and $\boldsymbol{\sigma}$ in equation \eqref{eqn:fokker-planck-sde}, initial set $\mathcal{A}_0 = \{\boldsymbol{z_0^k}\}_{k=1}^Q$.

\State{\textbf{Output:}} Center $\boldsymbol{O}$, Half-edge length $B$,  Numerical bounded domain $\bigotimes_{j=1}^d [O_j - B, O_j + B]$.

\State{Initialize $\mathbf{\Xi} = \emptyset$ (set of trajectory points)}
\For{$t = 0, 1, \dots, t_{terminal}-1$}

    \For{$k = 1, 2, \dots, Q$}

    \State{
    $\boldsymbol{z_{t + 1}^k} = \boldsymbol{z_{t}^k} + \boldsymbol{f}(\boldsymbol{z_{t}^k}, t) h + \boldsymbol{\sigma}(\boldsymbol{z_{t}^k}, t) \boldsymbol{\Delta \beta_{t}}$ where $\boldsymbol{\Delta \beta_{t}}$ is a sample from $N(0, h \boldsymbol{I}_{d \times d})$
    }
    \EndFor

    \State{Obtain the trajectory points at time $t+1$, $\mathcal{A}_{t+1} = \{\boldsymbol{z}_{t+1}^k\}_{k=1}^Q$}
     
    \If{$t > t_{burnin} - 1$}
    
    Add all points from $\mathcal{A}_{t+1}$ to $\mathbf{\Xi}$ 
    \EndIf

\EndFor

\State{Compute $\boldsymbol{O} = \frac{\sum_{\boldsymbol{z_i} \in \mathbf{\Xi}} \boldsymbol{z_i}}{\# \mathbf{\Xi}}$ and compute $B = C \max_{j \in \{1,...d\}}( \max_{\boldsymbol{z_i} \in \mathbf{\Xi}}(O_j - z_{ij}) )$.
}

\State{Obtain $\boldsymbol{O}$, $B$ and $\bigotimes_{j=1}^d [O_j - B, O_j + B]$.
}
\end{algorithmic}
\end{algorithm}


Training on a large bounded domain becomes expensive, especially in higher dimensions.
 We may refine the bounded domain using pre-trained neural network solutions. 
 First, we use the numerical support obtained above to train our tensor neural networks using the proposed methodology in this section and denote the solution as $p_N$. 
Second, we refine the numerical support by finding a small $r$, denoted as $r^*$ and  $r^*< B$, which corresponds to a bounded domain $\Omega^{*} = \bigotimes_{j=1}^d[O_j-r^*, O_j+r^*]$. $\Omega^{*}$ can still be denoted as 
$\Omega$ if no confusion arises. We desire that $\int_{\Omega^*} p_{N} \,d x$ is greater than a threshold,  $\vartheta $.  We then set the bounded domain as 
$\Omega=\bigotimes_{j=1}^d [O_j-r^*, O_j+r^*]$, which we call \textit{refined bounded domain or refined numerical support}. The process of finding $r^{*}$  is described in Algorithm \ref{alg:refine-support}.
\begin{algorithm}
\caption{Refining the numerical support}
\label{alg:refine-support}
\begin{algorithmic}
\State \textbf{Input:} pre-trained neural network $p_N$, initial half-edge length $r = B$, threshold $\vartheta$, center $\boldsymbol{O}$.
\State \textbf{Output:} refined half-edge length $r^{*}$, refined bounded domain $\Omega = \bigotimes_{j=1}^d [O_j - r^{*}, O_j + r^{*}]$.

\State \textbf{Step 1:} 
Partition the interval $(0, r]$ into a candidate set $S_n = \{r_i : i = 1, 2, \dots, n\}$.

\State \textbf{Step 2:} 
Find the smallest $r^{*} \in S_n$ such that:
$$
\int_{\bigotimes_{j=1}^d [O_j - r^{*}, O_j + r^{*}]} p_N(x) \, dx > \vartheta.
$$

\State \textbf{Step 3:} 
Set $r = r^{*}$ and update the refined domain $\Omega = \bigotimes_{j=1}^d [O_j - r^{*}, O_j + r^{*}]$.

\end{algorithmic}
\end{algorithm}


\zz{The hyperparameters $\vartheta$ will depend on the specific cases. Typically, we set $0.95 \leq \vartheta \leq 1.0$ since we expect that the integration of the exact solution $p^{*}$ on the refined domain is close to $1.0$.
See Section \ref{sec:numerical} for details and the candidate set $S_n$ in  \ref{sec:append-partition} for all the  test cases
in Section \ref{sec:numerical}.
}

\begin{table}[ht]
\centering
\begin{tabular}{|c|l|}
\hline
\textbf{Symbol} & \textbf{Explanation} \\ \hline
$\mathbf{\Xi}$ & Set of trajectory points from SDE simulations \\ \hline
$t_{terminal}$ & Terminal time for SDE simulations \\ \hline
$t_{burnin}$ & Burn-in step for SDE simulations\\ \hline
$Q$ & Number of trajectories\\ \hline \hline
$\Omega$       &  numerical support, bounded domain \\ \hline
$\boldsymbol{O}$          & Center of the numerical support, $\boldsymbol{O} = (O_1, O_2,..., O_d)$ \\ \hline
$r$            & Half-edge length of the numerical support\\ \hline
$B$          & Initial half edge length \\ \hline \hline
$p_N$          & Pre-trained neural network solution \\ \hline
$\vartheta$    & Threshold,  for  refining the numerical support \\ \hline
$r^*$          & Refined half-edge length for the numerical support \\ \hline
\end{tabular}
\caption{\color{black}Notations in Section \ref{sec:bounded domain}}
\end{table}

\subsection{Tensor neural network (TNN)} \label{sec:TNN}
We now introduce the structure of \textit{tensor neural network}(TNN) of the following form:
\begin{equation}\label{eq:KDE Estimator}
     p_N(\boldsymbol{x}; \Theta) =  \frac{1}{Z(\Theta)} \sum_{i=1}^N  c_i\bigotimes_{j=1}^d  k_{ij}(x_j;\Theta_{ij})\quad 
     \boldsymbol{x} = (x_1,x_2,\cdots,x_d) \in \Real^d,\, k_{ij}(x_j) \geq 0,\, c_i \geq 0, \quad x_j \in \Real. 
\end{equation}
Here $\Theta$ is the parameter set of TNN;
`$\bigotimes$' refers to the tensor product; and 
$Z(\Theta)$ is the normalization constant such that the integration of $p_N(x;\Theta)$ over $\Omega$ is 1 and 
\begin{equation}\label{eq:estimator-integral}
Z(\Theta) = \sum_{i=1}^N c_i \left( \prod_{j=1}^d \int_{[O_j- r, O_j + r]} k_{ij}(x_j;\Theta_{ij})  \,dx_j \right).
\end{equation}
Also,  $k_{ij}(x_j,\Theta_{ij})$'s are 
one-dimensional neural networks and nonnegative. 
We call  $N$ the rank of the tensor product networks and  $\bigotimes_{j=1}^d  k_{ij}(x_j;\Theta_{ij})$ the $i$-th sub tensor neural networks. 
%
%

In principle, any universal approximators can be applied here for the network $k_{ij}$. \rone{To keep the discussions brief}, we use two types of neural networks for $k_{ij}$: either feedforward neural networks or radial basis function networks. In the next two paragraphs, we will discuss the specifics of these two types of tensor neural networks.
\paragraph{Tensor Feed-Forward Networks}\label{sec:TFFN}
When the one-dimensional networks $k_{ij}$'s are all fully connected feed-forward neural networks, we call $p_N(x; \Theta)$ \textit{Tensor Feed-Forward Network} (abbreviated as TFFN).  Here we set $c_i = 1$ for all $i$.  
The trainable parameters are $\Theta_{ij}$ from the one-dimensional neural networks $k_{ij}$'s. We highlight several details of TFFN we use in simulations.
\begin{itemize}

\item (Number of Parameters):
As $k_{ij}$ is a fully-connected neural network with $l + 1$ hidden layers. The input dimension and output dimension of  $k_{ij}$ are both 1. Suppose that in each hidden layer we have  $w$ neurons. 
The order of number of parameters of $k_{ij}$ is $lw^2$. Since the number of $k_{ij}$ is $Nd$, the number of parameters of TFFN is $Ndlw^2$.

\item (Activation Functions): For each $k_{ij}$, activation functions of the hidden layers are all $tanh$. The activation function of the output layer is \textit{softplus}, which guarantees non-negative outputs. 

\item  (Normalization Constant): The normalization constant  $Z(\Theta)$ is computed with selected numerical integration. Details can be found in Section \ref{sec:methods-set-up-training}.

\end{itemize}

\paragraph{Tensor Radial Basis Function Networks}\label{sec:TRBFN}

When $k_{ij}$'s are all radial basis function networks, each $k_{ij}$ is an additive kernel, i.e., linear combination of radial basis functions ({RBF}) in the following form:
\begin{equation}\label{eq:additive-kernel-1d}
k_{ij}(x_j;\Theta_{ij}) = \sum_{\ell=1}^m \alpha_{ij}^{(\ell)} k_{ij}^{(\ell)}(| \frac{x_j - s_{ij}^{(\ell)}}{h_{ij}^{(\ell)}} | ),\quad \alpha_{ij}^{(\ell)} \geq 0, \, \sum_{\ell = 1}^m \alpha_{ij}^{(\ell)} = 1.  
\end{equation}
Here  $k_{ij}^{(\ell)}$ is a 1-dimensional radial basis function.
We call this tensor network as \textit{Tensor Radial Basis Function Network} (abbreviated as {TRBFN}).
\rone{
Here $m$ represents the number of basis and  $\Theta_{ij} = \{\boldsymbol{\alpha}_{ij} ,  \boldsymbol{s}_{ij}, \, \boldsymbol{h}_{ij} \}$, where for each $i$, $j$, $\boldsymbol{\alpha}_{ij}$, $\boldsymbol{s}_{ij}$ and $\boldsymbol{h}_{ij}$ are m-dimensional vectors. }
Specifically, $\alpha_{ij}^{(\ell)}$ is the $\ell$-th element in $\boldsymbol{\alpha}_{ij}$ which represents the coefficient of the $\ell$-th radial basis function; $h_{ij}^{(\ell)}$ is the $\ell$-th element in $\boldsymbol{h}_{ij}$ which represents the scale/bandwidth of $\ell$-th radial basis function $k_{ij}^{(\ell)}$; and  $s_{ij}^{(\ell)}$ is the $\ell$-th element in $\boldsymbol{s}_{ij}$ which represents the shifts/anchor points of $\ell$-th radial basis function $k_{ij}^{(\ell)}$.
{Unlike TFFN, $c_i$'s in TRBFN are trainable and $\sum_{i=1}^N c_i = 1$}. We highlight several details of TRBFN:

\begin{itemize}

\item (Radial Basis Functions): In theory, any positive definite kernels can be applied here. In this work, we are confined to the following radial basis functions for $k_{ij}^{\ell}$
in \eqref{eq:additive-kernel-1d}:  
\text{Gaussian} $\exp(-\abs{x}^2)$,  
\text{inverse multi-quadratic} $( (1 + \abs{x})^2 )^{-2.5}$,   and one of the \text{Wendland's functions}
  $\max{((1 - \abs{x})^3,0)} ( 3 \abs{x} + 1)$.

\item (Number of Parameters): In Equation \eqref{eq:additive-kernel-1d}, the numbers of $\alpha_{ij}^{(\ell)},\: s_{ij}^{(\ell)}\:$ and $h_{ij}^{(\ell)}$ are all $m$, for each $i,j$. Since the number of $k_{ij}$ is $Nd$, the number of parameters of TRBFN is $3Ndm$.

\item (Constraints of Parameters): Parameters in radial basis networks can be 
constrained for efficient training. 
In this work,  we impose constraints on the shifts $s_{ij}^{(\ell)} $and bandwidths $h_{ij}^{(\ell)} $ in TRBFN. \zz{The constraints are related to the half-edge length $r$ of the hypercube  $\Omega$ and we require}
\begin{equation}\label{eq:constraints-shifts}
|s_{ij}^{(\ell)} - O_j| < r,\qquad 
|h_{ij}^{(\ell)}| < |r - |s_{ij}^{(\ell)} - O_j||. 
\end{equation}
Intuitively, these constraints can help  TRBFN concentrate inside the bounded domain $\Omega$. These constraints will be implemented via penalization, see \eqref{eqn:trbfn-loss}.

\item (Normalization Constant): The normalization constant $Z(\Theta)$ is computed analytically due to the simple forms of the radial basis functions. 

\end{itemize}

\subsection{Loss function}\label{sec:methods-set-up-loss}
In the loss function, we incorporate the residual of the Fokker-Planck equation \eqref{eq:fokker-planck-steady} and the integration of the approximated solution being 1. The non-negativity of the solution can be accommodated by
applying a non-negative activation function in the last
layers of 1D feedforward networks or using non-negative radial basis functions, as described in the last subsection. 
At the continuous level, we use the following loss function $  \int_{\Omega} \Big({\mathcal{L}\displaystyle p_N(x;\Theta)\Big)}^2 \,dx.$ We  use uniformly sampled points on $\Omega$ to discretize the integral and thus obtain the following discrete loss:
\begin{equation}
\sum_{i=1}^{\# \,\text{samples}}\Big({\mathcal{L}\displaystyle p_N(x_i;\Theta)\Big)}^2. 
  \label{eqn: Loss}
\end{equation}

\paragraph{Boundary Condition}To ensure the vanishing value of TFFN on the boundary $\partial \Omega$, we multiply TFFN by the function 
\begin{equation}
\bigotimes_{j=1}^d \max{((\frac{1 - (x_j - O_j )^2}{r_j^2})^3,0)}\quad 
\label{eqn:trunc}
\end{equation}
\zz{
For TRBFN, we enforce the boundary condition weakly by penalization using the following loss \eqref{eqn:trbfn-loss}: 
\begin{equation}
  \begin{split}
  &
   \sum_{i=1}^{\# \,\text{samples}}\Big({\mathcal{L}\displaystyle p_N(x_i;\Theta)\Big)}^2 
  \\
  & + \mathcal{W}_1  \sum_{i,j,\ell}\left(\max(|s_{ij}^{(\ell)} - O_j| - r, 0 ) + \max(|h_{ij}^{(\ell)}| -  (|r - |s_{ij}^{(\ell)} - O_j|\,|), 0) \right)
  \\
&  + \mathcal{W}_2\sum_{i,j}( (k_{ij}(O_j + r)  + k_{ij}(O_j - r) ),
\end{split}
\label{eqn:trbfn-loss}
\end{equation}
where $\mathcal{W}_1> 0$ and $\mathcal{W}_2 > 0$ are the weights for different penalty terms. The values of $\mathcal{W}_1$ and $\mathcal{W}_2$ for different examples are presented in \ref{sec:append-Weights}. 
Specifically, we enforce the  boundary condition through  the penalty term in the third line of  
\eqref{eqn:trbfn-loss} while the second line of \eqref{eqn:trbfn-loss} represents the penalization of the constraints of parameters from Equation  \eqref{eq:constraints-shifts}. 
}

\subsection{Training settings} \label{sec:methods-set-up-training}
Both TRBFN and TFFN are implemented with the Python programming language using the JAX library. 
We use LION methods \cite{chen2024symbolic} 
for training both networks. 
\rone{We also compare the performance of different optimizers in the ablation studies in Section \ref{sec:ablation}}.
The number of epochs is $10^5$ \zz{for all examples  while batch sizes are detailed in \ref{sec:append-training}.}


\paragraph{Training Platform} We conduct all experiments of TFFN on an NVIDIA A100 GPU with 80 GB memory and all experiments of TRBFN on an NVIDIA A100 GPU with 40 GB memory.

\paragraph{Data Precision} \tw{For TFFN, we use double precision(float64) since we observe that TFFNs with double precision (float64) obtain better training result. For TRBFN, we use single precision (float32) since we observe that TRBFNs with double precision and single precision achieve the same results.}

\paragraph{Initialization of tensor neural networks} For TFFN, we apply Xavier Glorot's initialization to each feedforward neural network $k_{ij}$. For TRBFN, we initialize all $s_{ij}^{(\ell)}$ constrained within $[O_j - r, O_j + r]$ and normally distributed with mean $O_j$ and standard error $\sqrt{r}$.   Bandwidths/scaling factors $h_{ij}^{(\ell)}$ are initialized as $0.9 r$.

\paragraph{Computing $Z(\Theta)$} \zz{For the tensor feedforward networks (TFFN), we use the piecewise Gauss-Legendre quadrature rule to compute the integral $\int_{[O_j - r_j, O_j + r_j]}k_{ij}(x_j, \theta_{ij}) dx_j$ and then compute $\int_\Omega p_N(x;\theta) dx$. Specifically, we split $[O_j-r_j, O_j + r_j]$ into $M=16$ subintervals and use $16$  Gauss-Legendre quadrature points with corresponding weights in each subinterval. For tensor radial basis network (TRBFN), we use analytical forms of the integrations  of 
radial basis functions such that $\int_{[O_j- r, O_j + r]} k_{ij}(x_j;\theta_{ij}) dx_j = 1 $. }

\begin{rem}
Increasing the number of $N$ and $m$ can increase {TRBFN}'s approximation capacity. But increasing the number of $N$ seems more significant than increasing the number of $m$. 
 For our experiments in Example \ref{exm:2d-ring}, we found that increasing $N$ with a small number of $m$ ($m < 10$) can produce smaller losses 
and more accurate solutions
than increasing $m$ with a small number of $N$ ($N < 10$). 
Note that the number of parameters 
of TRBFN is at the order $3Ndm$. 
We thus choose a large $N$ and a small $m$ for TRBFN in Section \ref{sec:numerical}. 
For TFFN, we use a relatively small $N$ while we use the size of the one-dimensional feedforward neural networks to control the number of parameters. For example, if we use a networks with $l+1$ hidden layers and the architecture  $[1,\, w,\ldots\,w,1]$,
the number of parameters is at the order of 
$N d l w^2$.  In practical implementation, we need a balance between the rank $N$ and the number of radial basis functions or the size of feedforward neural networks, due to the tensor implementation of prevalent Python packages.
\end{rem}

\begin{algorithm}

\begin{algorithmic}
\caption{Solving the Fokker-Planck Equation}

\State{\textbf{Input:}} bounded domain $\Omega$, initial parameter $\Theta_0$, number of epochs $T$, the batch size $M$, learning rates.
\State{\textbf{Output:}} minimizer $\Theta^{*}$ and tensor neural networks.
\For{$t = 0, 1, \dots, T-1$}

\State{1. Sample $M$ points $D_t = \{x_i^{(t)}\}_1^M$ from uniform distribution on $\Omega$}

\State{2. Compute the gradient of loss function $\nabla L(D_t; \Theta_t )$ based on current parameters $\Theta_t$. Here $L(D_t;\Theta_t) $ represent the loss function with points $D_t$ and parameters of tensor neural networks $\Theta_t$. We use  the loss \eqref{eqn: Loss} for TFNN and the loss  \eqref{eqn:trbfn-loss} for TRBFN.}

\State{3. Update the parameter $\Theta_t$ with optimizer LION and gradient $\nabla L(D_t; \Theta_t )$. Obtain $\Theta_{t+1}$}
\EndFor
\State{
Obtain $\Theta^{*} = \Theta_T$ and the trained tensor neural networks.
}
\end{algorithmic}
\end{algorithm}

\section{Numerical results}\label{sec:numerical}
In this section, we consider the equation \eqref{eq:fokker-planck-steady} with the following drift
\begin{equation}\label{eq:processes}
   \boldsymbol{f}(\boldsymbol{x}) = - \frac{1}{2} \boldsymbol{D}(\boldsymbol{x}) \nabla \boldsymbol{H}(\boldsymbol{x}) +  \boldsymbol{g}(\boldsymbol{x}),\quad  g_{i}(\boldsymbol{x}) = \sum_{j=1}^d \frac{\partial}{\partial x_j} (\frac{1}{2}D_{ij}(\boldsymbol{x})).
\end{equation}
Then the solution to \eqref{eq:fokker-planck-steady}  reads 
\begin{equation}
    p(\boldsymbol{x})  = \frac{1}{\int_{\Real^d} \exp(-H(\boldsymbol{x}))\,dx} \exp(-H(\boldsymbol{x})).
\end{equation} 
We remark that the methodology in this work can be applied to the equation \eqref{eq:fokker-planck-steady} \zz{ with a more general drift} but this specific drift allows us to have a baseline solution to compare with.


To evaluate the numerical results, we use the following relative error 
\begin{equation}\label{eq:defn-relative-error}
   \frac{1}{n}\sum_{i=1}^n \abs{\frac{p^*(\boldsymbol{x}_i) - p_N(\boldsymbol{x}_i)}{p^*(\boldsymbol{x}_i)}}, \quad x_i\in  \Gamma_\varepsilon=
\set{\boldsymbol{x} \in \Gamma |p^*(\boldsymbol{x}) > \varepsilon}. 
\end{equation}
Here $p^*(\boldsymbol{x})$ is the exact solution and $p_N(\boldsymbol{x})$ is the numerical solution obtained from tensor neural networks.
Also, $\Gamma_{\varepsilon} \subset \Real^d$ is a set of interest and $\varepsilon>0$ and is specified in each example.  The testing points $\boldsymbol{x}_i$ are uniformly sampled from the region $\Gamma$ of interest.

We test TRBFN on five examples of steady-state Fokker-Planck equations. We have experimented with several choices of different RBFs in \eqref{eq:additive-kernel-1d}.  For simplicity, we 
denote by 
TRBFN($N$, $m$)  the tensor  radial basis function neural networks with  the rank $N$ in \eqref{eq:KDE Estimator} and the number of basis $m$ in \eqref{eq:additive-kernel-1d}. 
{We test TFFN on first three examples.} 
We also denote by TFFN($N$, $\cdot$) the tensor feedforward neural networks with 
the rank  $N$  in \eqref{eq:KDE Estimator} and the architecture of the 1D feedforward neural network ``$\cdot$". \tw{The feedforward network with $l + 1$ hidden layers, each of which has $w$ neurons,  will be denoted by [1 $w$ $\cdots$ $w$ 1].}

We list the bounded domain $\Omega$ in Table \ref{tab:bounded_domain} for all examples in this section. We take $C=1.1$ for all examples.
\begin{table}[!ht]
\centering
\begin{tabular}{|c|c|c|c|c|c|}
\hline
\textbf{Example} & \textbf{Dimension} & \text{B} & \textbf{Center(O)} & $\vartheta$ \\ 
\hline
4.1 & 2D & 2.1467 & (-0.0056, 0.0026) & N/A \\ 
\hline
4.2 & 4D & 2.6472 & (-0.0043, 0.1048,  0.0044, 0.0081) & 0.999\\ 
\hline
4.3 & 6D  & 1.5191 & origin & 0.990 \\ 
\hline
4.4 & 6D  & 5.2872 & (-0.0322, 0.0598, -0.0045, 0.0197, -0.0036, 0.0054) & 0.97 \\ 
\hline
4.5 & 10D & 2.8580 & 
(0.0017, -0.0016, -0.0026, -0.0025,
  0.0058,\ldots &\\
 & & & %
  -0.0030, 0.0, -0.0025,
 -0.0183, -0.0012) & N/A\\ 
\hline
\end{tabular}
\caption{{\color{black}The numerical supports
$\Omega= \bigotimes_{j=1}^d [O_j-B,O_j+B]$ for all examples in Section \ref{sec:numerical}.}}
\label{tab:bounded_domain}
\end{table}

\begin{exm}[2D, ring potential]\label{exm:2d-ring}
Consider the equation \eqref{eq:fokker-planck-steady}  with 
the following drift $\boldsymbol{f}$ in \eqref{eq:processes} using 
$H(x) =2 (x_1^2 + x_2^2 - 1)^2,$
and diffusion $\boldsymbol{D} = 2 \boldsymbol{I}_{2\times 2}$. 
This example was used in \cite{zhai2022DL4FP}. 
The potential function $H(x)$ has the level curve $x_1^2+x_2^2=1$, where the exact solution (the density function) reaches its mode.  
\end{exm}

The bounded domain we use in Example \ref{exm:2d-ring} is described in Table \ref{tab:bounded_domain} with center $\boldsymbol{O}$ and half-edge length \zz{$r = B = 2.1467$}. First, we test the \zz{TRBFN}.  Each $k_{ij}$ in \eqref{eq:additive-kernel-1d} for {TRBFN} only contains the Wendland radial basis functions. 
 We use TRBFN(5, 720) and TRBFN(1000, 3). 
 For both networks,   we plot the history of training loss versus training epochs in Figure \ref{fig:exm-2d-ring}. We observe that \text{TRBFN}(1000, 3) achieves significantly lower loss than \text{TRBFN}(5, 720) at $10^5$-th epochs:  $10^{-4}$ vs $10^{-2}$.  
 The smaller loss suggests a better approximation, as observed in Table 
 \ref{tbl:exm-2d-ring-trbf}, \zz{where 
 we} present the relative errors defined in  \eqref{eq:defn-relative-error}.  \zz{We uniformly sample $10^5$ test points in $\Gamma=[-2, 2]^2$ and we use} $\Gamma_\varepsilon$ with $\varepsilon=10^{-2},\,5 \times 10^{-2},10^{-1}$.
 From Table \ref{tbl:exm-2d-ring-trbf}, we observe  that using TRBFN(1000, 3) achieves the accuracy that is  
  at least one order of magnitude
 higher than using TRBFN(5, 720).  This suggests that  the large rank  ($N = 1000$) in the tensor radial basis function networks results in  larger 
 approximation capability than 
 a small rank ($N=5$). 
\begin{figure}[!ht]
    \centering
    \begin{subfigure}[b]{0.45\linewidth}
        \includegraphics[width=\linewidth]{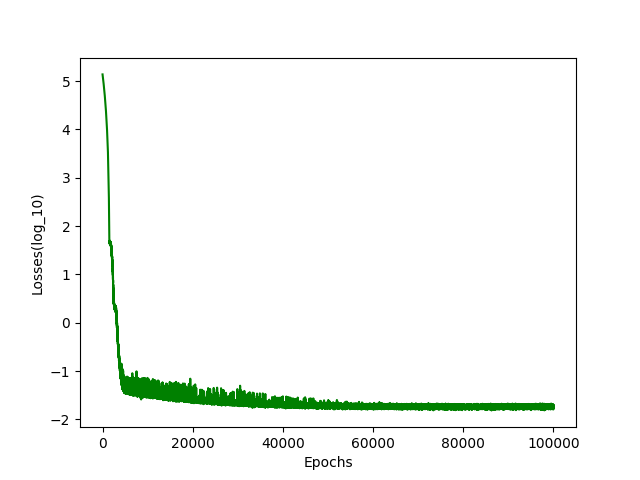} 
        \caption{\small{Loss history of TRBFN(5, 720)}}
    \end{subfigure}
    \hspace{0.05\linewidth}
    \begin{subfigure}[b]{0.45\linewidth}
        \includegraphics[width=\linewidth]{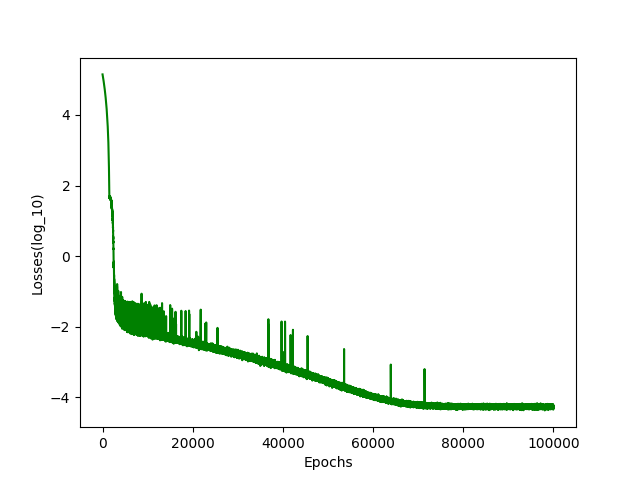} 
        \caption{\small{Loss history of TRBFN(1000, 3)}}
    \end{subfigure}
    \caption{Example \ref{exm:2d-ring}: history of loss function. In TRBFN, we use the Wendland kernel.}
    \label{fig:exm-2d-ring}
\end{figure}

\begin{table}[!ht]
    \centering
    \begin{tabular}{lcccc} 
        \toprule
        & & \multicolumn{3}{c}{density region/error } \\ 
        \cmidrule{3-5} 
        model & \# parameters (est.) & $\Gamma_{1 \times 10^{-2}}$ & $\Gamma_{5 \times 10^{-2}}$& $\Gamma_{1 \times 10^{-1}}$ \\ 
        \midrule
        TRBFN(1000,3) & 19000 & 0.0001 & 0.0001 & 0.0001 \\ 
        TRBFN(5,720) & 21605 &  0.0051 & 0.0029 & 0.0022 \\ \hline 
         TFFN(64,$[1\, 8\, 8\, \, 1]$) & 10240
 &0.0035 & 0.0025 & 0.0022\\
        TFFN(160,$[1\, 8\, 8\, 1]$) & 25600  & 0.0017 & 0.0011 & 0.0008  \\ 
        TFFN(5,$[1\, 32\, 32\, 32\, 1]$) & 21120
          & \tw{0.5297} & \tw{0.4686} & \tw{0.4751}  \\ 
        TFFN(10,$[1\, 32\, 32\, 32\, 1]$) & 42240  & 0.0070 & 0.0048 & 0.0045\\ 
        \tw{{TFFN}(1024, $[1\,8 \, 8 \,8 \,1]$)} & 294912 & \tw{0.0005} & \tw{0.0003} & \tw{0.0002}   \\  
        \hline
         $\#$  of  test points $n$ &  &44495 & 35392 & 26884 \\
        \bottomrule
    \end{tabular}
    \caption{\zz{ \tw{\small{Example \ref{exm:2d-ring}: average relative errors \eqref{eq:defn-relative-error} for different tensor neural networks. Here
    we sample  $10^5$ points from $\Gamma=[-2, 2]^2$.} } }
    }  \label{tbl:exm-2d-ring-trbf}
\end{table}

Second, we present results in Table \ref{tbl:exm-2d-ring-trbf}  from Tensor Feedforward Neural Networks. We observe that, for TFFN, a larger rank $N$ leads to better accuracy. 
For example, TFFN(160, [1, 8, 8, 1]) leads to smaller relative errors than TFFN(64, [1 8 8 1]).  
Compared to TRBFN, TFFN with rank 5 achieves \tw{two orders of magnitude lower accuracy than TRBFN with rank $N = 5$.}
\tw{We observe that, with deeper networks for each $k_{ij}$ and larger rank, TFFN(1024, [1 8 8 8 1]) achieves the same magnitude of accuracy as the TRBFN(1000, 3), and   the relative errors are  no larger than $0.05\%$ on the region $\Gamma_{10^{-2}}$. }

%




\begin{exm}[\zz{4D, solution with a single mode}] \label{exm:4.2-4dunimode}
Consider the drift 
\eqref{eq:processes} with the following 
potential function and the diffusion
\begin{align*}
 H(x)&= 3((x_1^4 - x_2)^2 + 2 x_2^2 ) + 
2 (x_3^2 - 0.3(x_3 x_4) + x_4^2 ),\\
\boldsymbol{D}  &= 2
\begin{bmatrix}
1 & 0 & 0 & 0\\
0 & 1 & 0 & 0\\
0 & 0 & 1.0 + V(x_3,x_4) & V(x_3,x_4)\\
0 & 0  & V(x_3,x_4) & 1.0 + V(x_3,x_4)
\end{bmatrix}, \quad V(x_3, x_4) = 0.1 x_3^2 x_4^2.
\end{align*}
The exact solution $p^*$  in  Example \ref{exm:4.2-4dunimode} 
has a mode at $(0, 0, 0, 0)$ while 
$H(x)$ exhibits complex interactions 
between $x_1$ and $x_2$,  and between $x_3$ and $x_4$. 
Specifically, the component $(x_1^4 - x_2)^2 + 2 x_2^2$ results in a high probability region focused around the line $x_1 =x_2$.
Moreover, steep gradients from 
$H(x)$ appear along the line $x_1=x_2$
near the origin, contributing to the concentration of the density function in this region.
The components involving $x_3$ and $x_4$ in $H(x)$ 
lead to the density function concentrated around the point 
$(0, 0, 0, 0)$.
\end{exm}

\zz{The numerical support we use in Example 4.2 is described in Table \ref{tab:bounded_domain} with center $O$ and half-edge length $r = B = 2.6472$.}
We first test TRBFN and TFFN on the numerical support \zz{$\bigotimes_{j=1}^4[O_j - 2.6472, O_j + 2.6472]$}, and then refine this bounded domain with TRBFN and test TFFN on the refined bounded domain.

\begin{figure}[!ht]
    \centering
    \begin{subfigure}[b]{0.45\linewidth}
 \includegraphics[width=\linewidth]{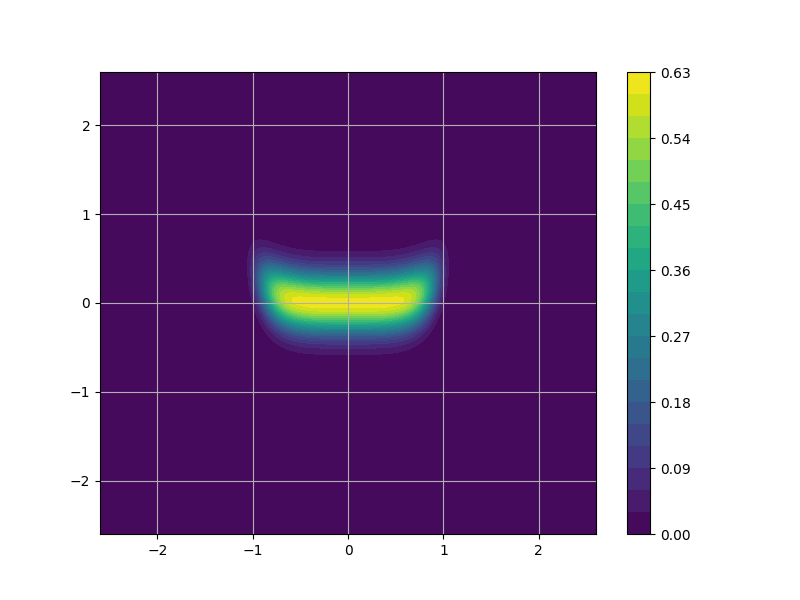} 
        \caption{\small{$(x_1, x_2, 0, 0)$}}
    \end{subfigure}
    \hspace{0.05\linewidth}
    \begin{subfigure}[b]{0.45\linewidth}
       \includegraphics[width=\linewidth]{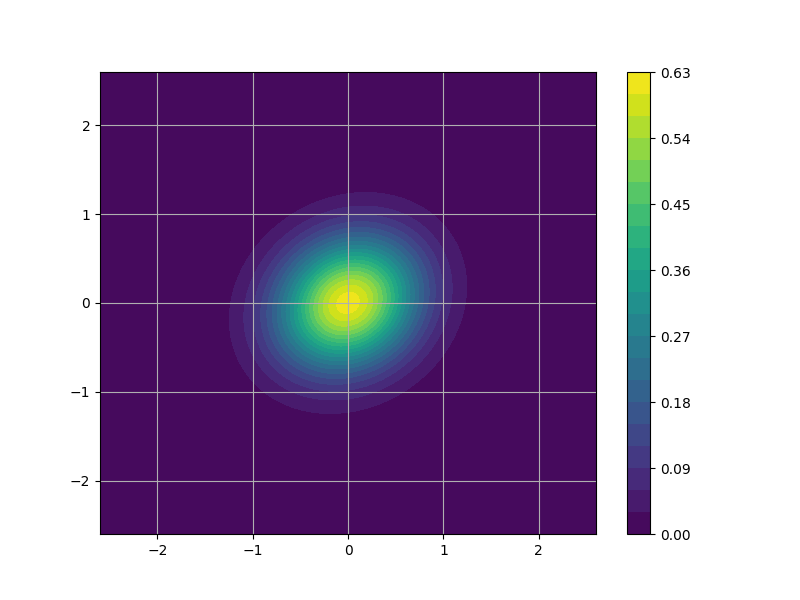} 
        \caption{\small{$(0, 0, x_3, x_4)$}}
    \end{subfigure}
    \caption{Contour maps of projections of  the exact solution $p^*(x)$ in 
    Example \ref{exm:4.2-4dunimode}}
\end{figure}

First, we test TRBFN on \zz{$\bigotimes_{j=1}^4[O_j - 2.6472, O_j + 2.6472]$}. 
All TRBFNs in this example only contain Wendland radial basis functions. We report in Table \ref{tbl:support-4d-uni-mode-trbfn-tfnn} the average relative errors. \zz{We observe that} all the  relative average errors from 
 TRBFN are less than 1\% on $\Gamma_{10^{-2}},\Gamma_{5\times 10^{-2}},\Gamma_{10^{-1}}$, except 
 TRBFN (64,27) on $\Gamma_{10^{-2}}$ \zz{where 
 $\Gamma=[-1,1]^4$}.
 Comparing results in Table \ref{tbl:support-4d-uni-mode-trbfn-tfnn} among TRBFNs, we conclude that a large rank and small $m$(number of basis functions) in TRBFN typically result in reduced error.

 Then, we test TFFN on the initial bounded domain \zz{$\bigotimes_{j=1}^4[O_j - 2.6472, O_j + 2.6472]$}. \tw{We observe that TFFN(128, [1 8 8 1]) and TFFN(64, [1 8 8 1]) have relative error greater than $10\%$ on $\Gamma_{10^{-2}}$.}

Next, we test TFFN on a smaller numerical support.  We refine the initial bounded domain by applying \zz{Algorithm \ref{alg:refine-support}} in Section \ref{sec:bounded domain}, where we set $\vartheta= 0.999$ and obtain 
$\bigotimes_{j=1}^4[O_j - 2.0, O_j + 2.0]$ with TRBFN trained on the initial bounded domain.  We report in Table 
\ref{tbl:support-4d-uni-mode-trbfn-tfnn} the average relative errors for  TFFN$(64, \,[1\, 8\, 8\, 1])$ and TFFN $(128, \,[1\, 8\, 8\, 1])$  trained on $\bigotimes_{j=1}^4[O_j - 2.0, O_j + 2.0]$. Both TFFNs achieve similar performance. The relative errors are below \tw{7\% on $\Gamma_{10^{-2}},\Gamma_{5\times 10^{-2}},\Gamma_{10^{-1}}$. A smaller numerical support allows TFFNs to achieve higher accuracy.}

\begin{table}[!ht]
    \centering
    \begin{tabular}{lccccc} 
        \toprule
       & & & \multicolumn{3}{c}{density region/error } \\ 
        \cmidrule{4-6} 
         model & $r$ & $\#$ parameters (est.) & $\Gamma_{1 \times 10^{-2}}$ & $\Gamma_{  5\times 10^{-2}}$ &  $\Gamma_{1\times 10^{-1}}$ \\ 
        \midrule
           TRBFN(1000, 3) &  $r = B = 2.6472$ & 37000 &  0.0045 & 0.0025 & 0.0016 \\ 
           TRBFN(64, 27) & $r = B = 2.6472$ & 20800& 0.0108 & 0.0060 & 0.0040 \\ 
           TRBFN(128, 27) & $r = B = 2.6472$ & 41600& 0.0069 & 0.0038 & 0.0026 \\ 

        \hline
        TFFN(64, [1 8 8 1]) & $r = B = 2.6472$  & 20480 & 0.1533 & 0.1187 & 0.1064 \\  
         TFFN(128, [1 8 8 1]) & $r = B = 2.6472$  & 40960  & 0.2043 & 0.1432 & 0.1226 \\ 
         
         TFFN(64, [1 8 8 1]) &  $r = 2.0$  & 20480  & 0.0696 & 0.0441 & 0.0329 \\  
         TFFN(128, [1 8 8 1]) &  $r = 2.0$  & 40960  & 0.0437 & 0.0257 & 0.0172 \\  \hline
         $\#$  of  test points $n$ & & & 51834 & 29950 & 18413 \\
        \bottomrule
    \end{tabular}
    \caption{ \tw{ \small{Example \ref{exm:4.2-4dunimode}: average relative errors \eqref{eq:defn-relative-error} for different tensor neural networks.  All $1\times10^5$ test points $x_i$ are sampled uniformly from $\Gamma=[-1, 1]^4$.} }  }
    \label{tbl:support-4d-uni-mode-trbfn-tfnn}
\end{table}

\begin{exm}[{\zz{6D, solution with a single mode}}]
\label{exm:6d-unimode}
Consider the drift 
\eqref{eq:processes} with the following 
potential function $H(x)=3 ((x_1^4 - x_2)^2 + 2 x_2^2 + (x_3^4 - x_4)^2 + 2 x_4^2 + (x_5^4 - x_6)^2 + 2 x_6^2) )$ and the diffusion  
$\boldsymbol{D} = 2 \boldsymbol{I}_{6\times 6}$.

In Example \ref{exm:6d-unimode}, the potential $H(x)$ has three specific terms of the form
$(x_{i}^4 - x_{i + 1})^2$, $i = 1,\, 3,\, 5$. The density function is more concentrated and has steeper gradients 
in regions of high density compared to Example \ref{exm:4.2-4dunimode} of four dimensions. \end{exm}

\zz{The initial bounded domain in Example \ref{exm:6d-multi-mode} is described in Table \ref{tab:bounded_domain} with center $\boldsymbol{O}$ and half-edge length $r = B = 1.5191$.} We will test TRBFN on  \zz{$\bigotimes_{j=1}^6[O_j - 1.5191, O_j + 1.5191]$} and test TFFN on the refined \zz{numerical support obtained using Algorithm \ref{alg:refine-support}.}

We first test this example with TRBFN(800, 3),  whose $k_{ij}$ only contains Wendland radial basis functions, on \zz{$\bigotimes_{j=1}^6[O_j - 1.5191, O_j + 1.5191]$}. The results are reported in Table \ref{tbl:small-support-6d-uni-mode-tnn}.  %
\zz{Using Algorithm \ref{alg:refine-support} with $\vartheta = 0.99$ and the TRBFN(800, 3) trained on initial bounded domain}, we obtain a smaller support $\bigotimes_{j=1}^6[O_j - 1.2, O_j + 1.2]$.

\tw{
We report the relative average errors in Table \ref{tbl:small-support-6d-uni-mode-tnn} for TRBFN over larger support and for TFFN  over the smaller support.} While both networks can capture the key features of the exact solution, TRBFN admits smaller relative errors over regions present in the table. 
For TFFN, we observe that increasing rank $N$ leads to reduced errors. 
We also observe in \zz{Table \ref{tbl:small-support-6d-uni-mode-tnn}} that deeper and wider neural networks give better results on the same support.

\begin{table}[!ht]
    \centering
    \begin{tabular}{lccccc} 
        \toprule
        & & & \multicolumn{3}{c}{density region/error } \\ 
        \cmidrule{4-6} 
         model &  $r$  & \# parameters(est.) & $\Gamma_{5\times 10^{-2}}$  & $\Gamma_{2.5\times 10^{-1}}$ & $\Gamma_{5\times 10^{-1}}$ \\ 
        \midrule
       TRBFN(800, 3) & $r = B = 1.5191$ & 44000 & 0.0891 & 0.0543 & 0.0427 \\ 
       \midrule
          {TFFN}(10, [1 32 32 32 1]) & refined $r = 1.2$   & 126720 & 0.1855 & 0.1132  &  0.0894  \\ 
         {TFFN}(20, [1 32 32 32 1])  & refined $r = 1.2$  & 253440 & 0.1001  &  0.0602  &  0.0479 \\ 
         {TFFN}(128, [1 8 8 1]) & refined $r = 1.2$   & 61440 & 0.3762 &    0.2646 & 0.2116   \\ 
         {TFFN}(20, [1 8 8 8 1]) & refined $r = 1.2$  & 17280& 0.2763 & 0.1817 & 0.1427 \\ 
         \midrule
           $\#$  of  test points $n$ & & & 34705 &     7796 &    1926    \\
        \bottomrule
    \end{tabular}
    \caption{\small{
    Example \ref{exm:6d-unimode}: average relative errors \eqref{eq:defn-relative-error} for different tensor neural networks. All $5 \times 10^5$ test points $x_i$ are sampled uniformly from $\Gamma = [-1, 1]^6$.
    }}
    \label{tbl:small-support-6d-uni-mode-tnn}
\end{table}

\vskip 20pt 

Those results from Examples \ref{exm:4.2-4dunimode} and \ref{exm:6d-unimode} demonstrate the capability of TRBFN and TFFN.
However, the exact solutions in Examples  \ref{exm:4.2-4dunimode} and \ref{exm:6d-unimode} have one mode. In Example \ref{exm:6d-multi-mode}, we present a density function with multiple modes.

\begin{exm}[\zz{6D, solution with multiple modes}]
\label{exm:6d-multi-mode}
Consider the drift 
\eqref{eq:processes} with the following 
potential function and the diffusion
\begin{align*}
 & H = 2(x_1^2 + x_2^2 + x_3^2 + 0.5(x_1 x_2 + x_1 x_3 + x_2 x_3)) - 
\ln(x_1^2 + 0.02) - \ln(x_2^2 + 0.02) \\
& + 0.5 (x_4^2 + x_5^2 + x_6^2 + 0.2(x_4 x_5 + x_4 x_6 + x_5 x_6) ), \quad \boldsymbol{D} = 2 \boldsymbol{I}_{6\times 6}.
\end{align*}
\end{exm}

The potential function
$H(x)$ in Example \ref{exm:6d-multi-mode} has three way interactions: interaction among $(x_1, x_2, x_3)$ and  interaction among $(x_4, x_5, x_6)$. 
The density function has four modes and each mode is not close to the rest. In this example, the training batch size is 40000.

\begin{figure}[!ht]
    \centering
    \begin{subfigure}[b]{0.45\linewidth}
        \includegraphics[width=\linewidth]{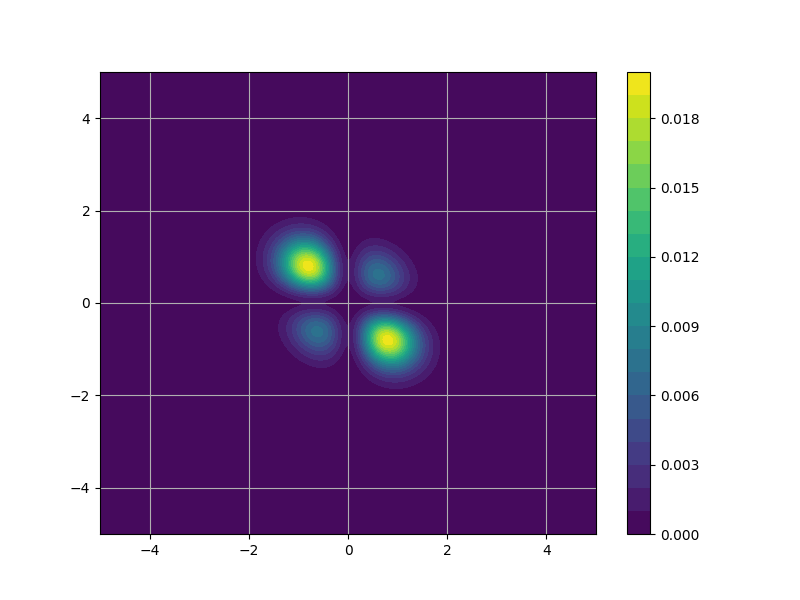} 
        \caption{\small{$(x_1, x_2, 0, 0, 0, 0)$}}
    \end{subfigure}
    \hspace{0.04\linewidth}
    \begin{subfigure}[b]{0.45\linewidth}
        \includegraphics[width=\linewidth]{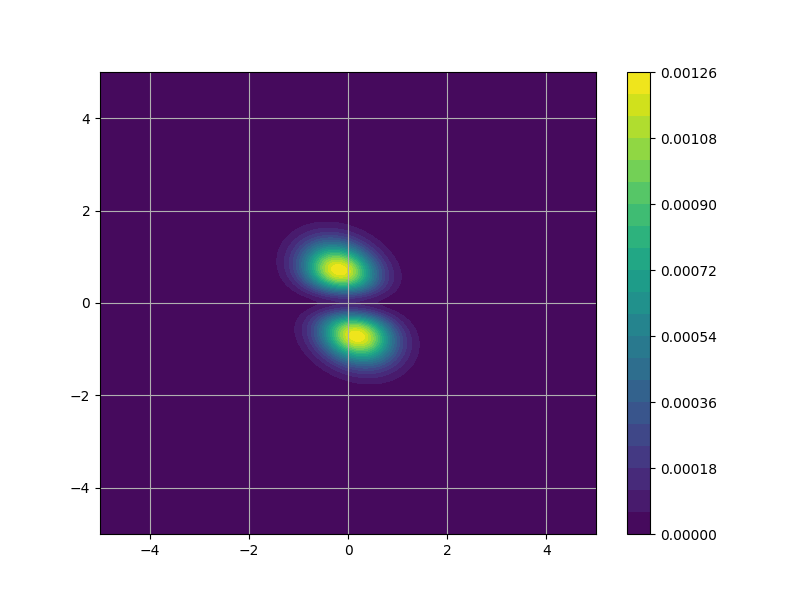} 
        \caption{\small{Example 4.4: $(x_1, 0, x_3, 0, 0, 0)$}}
    \end{subfigure}
    \caption{\small{Contour maps of  projections of the exact solution in Example \ref{exm:6d-multi-mode}}}\label{exm:contour-6d-multi-mode-exact}
\end{figure}

\zz{The initial bounded domain we use in Example \ref{exm:6d-multi-mode} is described in Table \ref{tab:bounded_domain} with center $\boldsymbol{O}$ and half-edge length $r = B = 5.2872$.} We test TRBFN on  $\bigotimes_{j=1}^6[O_j - 5.2872, O_j + 5.2872]$ and on \zz{a refined numerical support.  In this example, we only use   Wendland radial basis functions in TRBFN.}

We first test TRBFN on \zz{$\bigotimes_{j=1}^6[O_j - 5.2872, O_j + 5.2872]$.  The volume of  numerical support
 is large and the solution on most of the region $\Omega$ is small.
 We observe that the loss functions oscillate with epochs; see Figure \ref{fig:loss-trbfn-6d-mutiple-mode} where we use the tensor radial basis function network {TRBFN}(800, 6).
 We observe} in Table  \ref{tbl:6d-multiple-mode} that the relative errors within different density regions are under $30\%$. We present the contour map of projections of the numerical solution in Figure \ref{fig:contour-6d-mutiple-mode-est} and observe that the numerical solution captures the modes of the exact solution.

\begin{minipage}{0.5\textwidth}
 \centering
        \includegraphics[width= 1\textwidth]{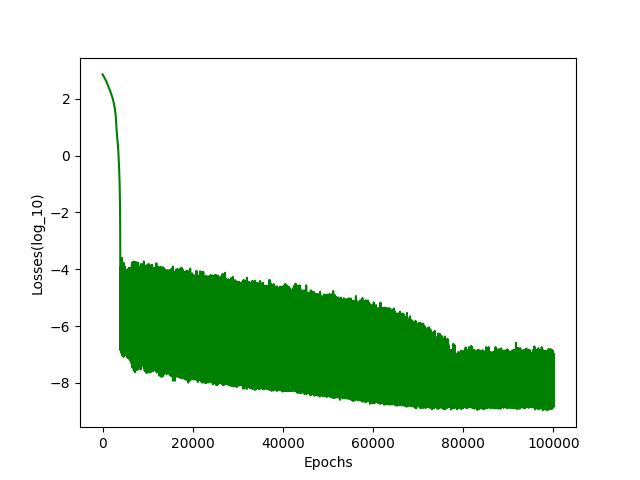} 
        \captionof{figure}{\small{Example \ref{exm:6d-multi-mode}: Loss history, TRBFN(800, 6). }}
        \label{fig:loss-trbfn-6d-mutiple-mode}
 \end{minipage}
\begin{minipage}{0.45\textwidth}
 \mbox{}
\vskip 10pt 
    \centering
    \begin{tabular}{lcccc} 
        \toprule
        & & \multicolumn{3}{c}{density region/error } \\ 
        \cmidrule{3-5} 
        &  & $\Gamma_{2\times 10^{-4}}$ & $\Gamma_{10^{-3}}$ & $\Gamma_{5\times 10^{-3}}$ \\ 
        \midrule
          & TRBFN(800, 6)  & 0.2550 & 0.2479 & 0.2589\\
        & \# of  test points $n$ & 91712 &     26334 &      2121 \\
        \bottomrule
    \end{tabular}
    \captionof{table}{\small{Example \ref{exm:6d-multi-mode}: average relative errors \eqref{eq:defn-relative-error} for TRBFN(800, 6) trained on $\Omega =  \bigotimes_{j=1}^6[O_j - 5.2872, O_j + 5.2872]$.
     All $5\times10^5$ test points $x_i$ are sampled uniformly from $\Gamma = [-2, 2]^6$.} The number of parameters of TRBFN(800, 3) is 44000.}
    \label{tbl:6d-multiple-mode}
\end{minipage}

\begin{figure}[!ht]
    \centering
    \begin{subfigure}[b]{0.45\linewidth}
        \includegraphics[width=\linewidth]{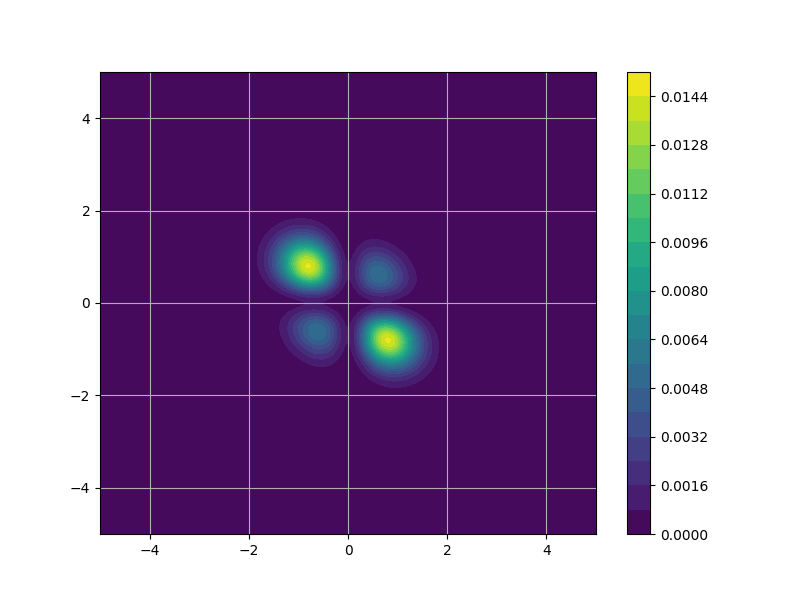} 
        \caption{\small{$(x_1, x_2, 0, 0, 0, 0)$}}
    \end{subfigure}
    \hspace{0.04\linewidth}
    \begin{subfigure}[b]{0.45\linewidth}
        \includegraphics[width=\linewidth]{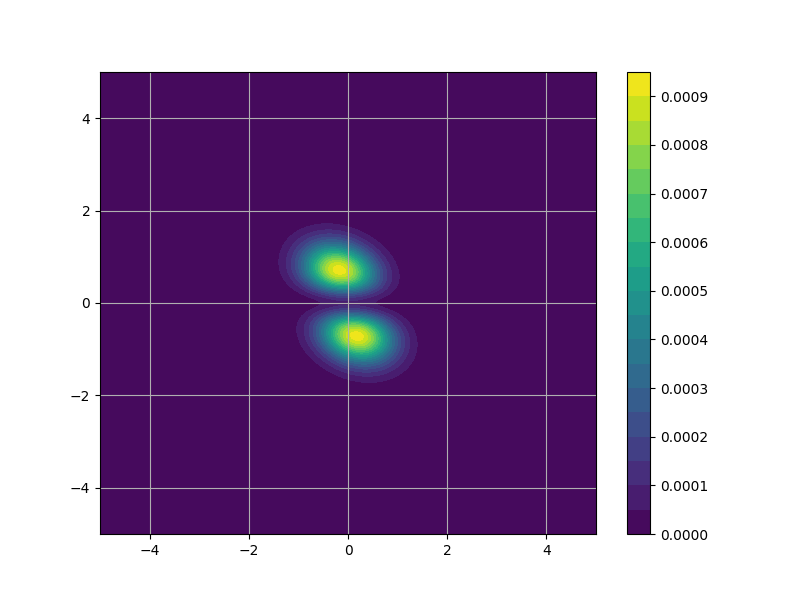} 
        \caption{\small{Example 4.4: $(x_1, 0, x_3, 0, 0, 0)$}}
    \end{subfigure}
    \caption{\small{Contour maps of  projections of the numerical solution in Example \ref{exm:6d-multi-mode}}}\label{fig:contour-6d-mutiple-mode-est}
\end{figure}

Next, we refine the numerical support as \zz{using Algorithm \ref{alg:refine-support}}.  With the trained {TRBFN}(800, 6), we obtain $\bigotimes_{j=1}^6[O_j - 4.4, O_j + 4.4]$ such that the integral of the TRBFN over the region is larger than the threshold $\vartheta = 0.97$.  Then the numerical support is refined to be  
$\bigotimes_{j=1}^6[O_j - 4.4, O_j + 4.4]$. 
We apply {TRBFN}(800, 6) on the refined support. In Table  \ref{tbl:small-support-6d-multiple-mode}, we report the average relative error, which is less than $5\%$ in the region $\Gamma_{10^{-3}}$ and is less than $10\%$ in   $\Gamma_{2\times 10^{-4}}$\zz{, where $\Gamma=[-2,2]^6$}.

\begin{table}[!ht]
    \centering
    \begin{tabular}{lcccc} 
        \toprule
        & \multicolumn{3}{c}{density region/error } \\ 
        \cmidrule{5-5} 
         model &  $\Gamma_{2\times 10^{-4}}$ & $\Gamma_{  10^{-3}}$ & $\Gamma_{5\times 10^{-3}}$ \\ 
        \midrule
         TRBFN(800, 6)  & 0.0867 & 0.0560 & 0.0365 \\ 
         $\#$  of  test points $n$ & 91712 &     26334 &      2121 \\
        \bottomrule
    \end{tabular}
    \caption{\small{Example \ref{exm:6d-multi-mode}: average relative errors \eqref{eq:defn-relative-error} for TRBFN(800, 6) trained on $\bigotimes_{j=1}^6[O_j - 4.4, O_j + 4.4]$. All $5\times10^5$ test points $x_i$ are sampled uniformly from $\Gamma = [-2, 2]^6$.}  The number of parameters of TRBFN(800, 3) is 44000(est.)}
    \label{tbl:small-support-6d-multiple-mode}
\end{table}

\begin{exm}[\zz{10D, solution with two modes}]
\label{exm:10d-multi-mode}
Consider the drift 
\eqref{eq:processes} with the following 
potential function and the diffusion
\begin{align*}
  H &= 2.5 (x_1^2 + x_2^2 + x_3^2 + 0.1 (x_1 x_2 + x_1 x_3 + x_2 x_3))
 + 2.0 (x_4^2 + x_5^2 + x_6^2 + 0.2(x_4 x_5 + x_4 x_6 + x_5 x_6) )\\
& + 3.0 (x_7^2 + x_8^2 - 0.01(x_7 x_8) ) + 3.0 (x_9^2 + x_{10}^2 - 0.01(x_9 x_{10}) ) - \ln(2 x_9^2 + 0.02), \quad D = 2 I_{10 \times 10}.
\end{align*}
In this example, the variables $(x_1, x_2, x_3), (x_4, x_5, x_6),
(x_7, x_8)$ and $(x_9, x_{10})$ have distinct interactions and the exact solution has two modes. The density function is highly concentrated around two modes.
\end{exm}

We test the capability of TRBFN to solve this 10D problem on different bounded domains $\Omega$. \zz{In this example, all TRBFNs only contain Wendland radial basis functions.}
\zz{We can not achieve any meaningful results with the bounded domain $\bigotimes_{j=1}^{10} [O_j - B, O_j + B]$, $O$ and $B = 2.8580$, which is obtained using Algorithm \ref{alg:obtain-bounded-domain}}. The relative error is over $99.99\%$.
\rone{To obtain a smaller bounded domain $\Omega$, we choose anisotropic half-edge length, i.e., $ \bigotimes_{j=1}^{10}[O - r_j, O + r_j]^{10}$ and 
$r_j = \max(O_j - z_{ij})$,
%
where $z_i$ are generated from simulated SDE trajectories from Algorithm \ref{alg:obtain-bounded-domain},  and $j$ denotes the $j$-th dimension. We obtain that $r_1 = 2.2911, r_2 = 2.1985, r_3 = 2.17, r_4 = 2.4365, r_5 = 2.622, r_6 = 2.3835, r_7 = 2.2343, r_8 = 1.875, r_9 = 2.1955, r_{10} = 2.0016$.} We test TRBFN(800, 3) on $\bigotimes_{j = 1}^{10} [O_j - r_j, O_j + r_j]$.  We report the result in Table \ref{tab:10D-rj-omega} and present the contour map of projections of the numerical solution and accurate solutions in Figure \ref{fig:contour-10d-mutiple-mode-est}.

\begin{table}[!ht]
    \centering
    \begin{tabular}{lccccc} 
        \toprule
        & & & \multicolumn{3}{c}{density region/error } \\ 
        \cmidrule{4-6} 
         model &   bounded domain & L2 error & $\Gamma_{1\times 10^{-3}}$  & $\Gamma_{1\times 10^{-2}}$ & $\Gamma_{1\times 10^{-1}}$ \\ 
        \midrule
          {TFFN}(800, 3) & $[O_j - r_j, O_j + r_j]^{10}$  & $8.8\times 10^{-3}$ & 0.7608 & 0.7657  &  0.7692  \\ 
         \midrule
           $\#$  of  test points $n$ & & & 390830 & 96024 & 211    \\
        \bottomrule
    \end{tabular}
    \caption{\small{
     Example \ref{exm:10d-multi-mode}: average relative errors \eqref{eq:defn-relative-error} for different tensor neural networks. All $5 \times 10^5$ test points $x_i$ are sampled uniformly from $\Gamma = [-0.7, 0.7]^{10}$.
    }}
    \label{tab:10D-rj-omega}
\end{table}

\begin{figure}[!ht]
    \centering
    \begin{subfigure}[b]{0.35\linewidth}
        \includegraphics[width=\linewidth]{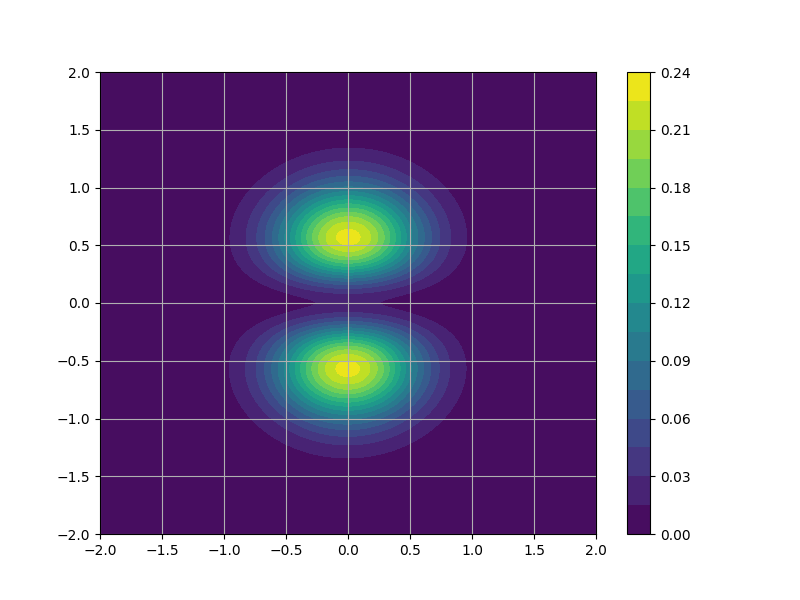} 
        \caption{accurate solution projected to \\ \small{$(0, 0, 0, 0, 0, 0, 0, 0, x_9, x_{10})$}}
    \end{subfigure}
    \hspace{0.04\linewidth}
    \begin{subfigure}[b]{0.35\linewidth}
        \includegraphics[width=\linewidth]{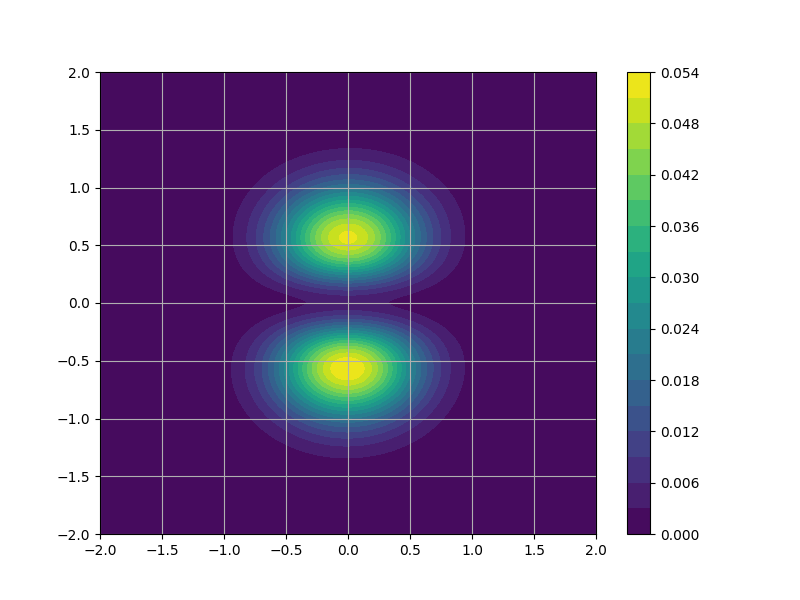} 
        \caption{numerical solution projected to\\ \small{$(0, 0, 0, 0, 0, 0, 0, 0, x_9, x_{10})$}}
    \end{subfigure}
    \caption{\small{Contour maps of projections of the numerical solution and accurate solution in Example \ref{exm:10d-multi-mode}}}\label{fig:contour-10d-mutiple-mode-est}
\end{figure}

While the relative errors are around $75\%$ for different density regions, the contours of the projection demonstrated that the TRBFN captures the modes of the accurate solution. We compute integrations of the trained TRBFN on the $\bigotimes_{j = 1}^{10}[O_j - r, O_j + r]$ where $r$ takes values in grids $\{0.3, 0.6, 0.9, 1.2, 1.5, 1.8, 2.1, 2.4\}$. We report the integration of the trained TRBFN in Table \ref{tab:tenD-integration}.

\begin{table}[ht]
\centering
\begin{tabular}{c|c|c|c|c|c|c|c|c}
\hline
\textbf{Model} & $r = 0.3$ & $r = 0.6$ & $r = 0.9$ & $r = 1.2$ & $r = 1.5$ & $r = 1.8$ & $r = 2.1$ & $r = 2.4$ \\ 
\hline
TRBFN(800, 3) &0.00 & 0.017  & 0.132 & 0.277& 0.476 & 0.784 & 0.967 & 0.999  \\
\hline
\end{tabular}
\caption{Example \ref{exm:10d-multi-mode}:  integration of  trained TRBFN(800, 3) over numerical supports with different half-edge length $r$.}
\label{tab:tenD-integration}
\end{table}

We also test TRBFN on the bounded domain $\bigotimes_{j=1}^{10}[O_j - 2.1, O_j + 2.1]$ and $\bigotimes_{j=1}^{10}[O_j - 1.8, O_j + 1.8]$ \zz{and report relative errors}  in Table \ref{tab:10D-1_8_2_1-results}. We observe that TRBFN trained on $\bigotimes_{j=1}^{10}[O_j - 2.1, O_j + 2.1]$ achieves relative errors less than $50\%$ on different density regions.  TRBFN trained on $\bigotimes_{j=1}^{10}[O_j - 1.8, O_j + 1.8]$ achieves relative errors less than $5\%$ on different density region. This demonstrates the capability of the \zz{TRBFN} in solving 10D Fokker-Planck equation. 
 However, we need to find a smaller bounded domain $\Omega$ where the most probability mass is concentrated on. This will be a future research topic.
 
\begin{table}[!ht]
    \centering
    \begin{tabular}{lccccc} 
        \toprule
        & & & \multicolumn{3}{c}{density region/error } \\ 
        \cmidrule{4-6} 
         model &   numerical support & L2 error & $\Gamma_{1\times 10^{-3}}$  & $\Gamma_{1\times 10^{-2}}$ & $\Gamma_{1\times 10^{-1}}$ \\ 
        \midrule
          {TFFN}(800, 3) & $\bigotimes_{j=1}^{10}[O_j - 2.1, O_j + 2.1]$  & $5.7\times 10^{-3}$ & 0.4831 & 0.4922 & 0.4904  \\ 
          {TFFN}(800, 3) & $\bigotimes_{j=1}^{10}[O_j - 1.8, O_j + 1.8]$  & $5.4\times 10^{-4}$ & 0.0422 & 0.0426 & 0.0472  \\ 
         \midrule
           $\#$  of  test points $n$ & & & 390830 & 96024 & 211    \\
        \bottomrule
    \end{tabular}
    \caption{\small{
     Example \ref{exm:10d-multi-mode}: average relative errors \eqref{eq:defn-relative-error} for different tensor neural networks. All $5 \times 10^5$ test points $x_i$ are sampled uniformly from $\Gamma = [-0.7, 0.7]^{10}$.
    }}
    \label{tab:10D-1_8_2_1-results}
\end{table}

\section{Comparison and ablation studies}\label{sec:ablation}
\rone{
In this section, we perform a comparison study with the state-of-the-art method for solving Equation \eqref{eq:fokker-planck-steady} and ablation studies on different choices of training optimizers, boundary conditions, choices of radial basis functions of TRBFN, and computation precision. 

\subsection{Comparison with a data-driven method}
\rone{
We compare our method with the data-driven method from  \cite{zhai2022DL4FP} for Example \ref{exm:6d-unimode}. 
The data-driven method in \cite{zhai2022DL4FP} uses feedforward neural networks to approximate the solution to the Fokker-Planck equation \eqref{eq:fokker-planck-steady}. 
The loss function therein is similar to \eqref{eqn: Loss}, but the training points for $x_i$ are generated from the corresponding SDE \eqref{eqn:fokker-planck-sde} (see Algorithm 2 in \cite{zhai2022DL4FP}). Also, an extra loss term is added - the mean-square difference between the reference solution and the feed-forward neural network solution at reference points. The reference solutions are Monte Carlo approximations from numerical simulations at reference points (see details in \cite{zhai2022DL4FP}).

}
\begin{table}[!ht]
    \centering
    \begin{tabular}{lccccc} 
        \toprule
        & & \multicolumn{3}{c}{density region/error } \\ 
        \cmidrule{2-4} 
         model &   $\Gamma_{5\times 10^{-2}}$  & $\Gamma_{2.5\times 10^{-1}}$ & $\Gamma_{5\times 10^{-1}}$ \\ 
        \midrule
          {TRBFN}(800, 3) & 0.0891 & 0.0543 & 0.0427  \\ 
         {TFFN}(20, [1, 32, 32, 32, 1]) & 0.1001  &  0.0602  &  0.0479 \\ 
         {NN data driven}  &  0.1166 & 0.0694 & 0.0788 \\ 
         \midrule
           $\#$  of  test points $n$ & 34705 &     7796 &    1926    \\
        \bottomrule
    \end{tabular}
    \caption{\small{{ 
    Comparison with data-driven methods in \cite{zhai2022DL4FP} for Example \ref{exm:6d-unimode}: average relative errors \eqref{eq:defn-relative-error} for different tensor neural networks. All $5 \times 10^5$ test points $x_i$ are sampled uniformly from $\Gamma = [-1, 1]^6$.}
    }}
    \label{tab:compared_data_driven}
\end{table}

For the method from \cite{zhai2022DL4FP}, we run $2.5 \times 10^6$ steps of SDE using the Euler-Maruyama method with $4 \times 10^5$ trajectories and a step size $0.001$. We obtain $4 \times 10^5$ points at the terminal time and use them as both training and reference points. 
%
The feedforward neural network consists of four hidden layers, each with 128 neurons, and has an input dimension of 6 and an output dimension of 1.
We train this neural network using Algorithm 1 in \cite{zhai2022DL4FP} and we use the same number of epochs for our method and the method from \cite{zhai2022DL4FP}.

We report the results in Table \ref{tab:compared_data_driven}. 
TRBFN is trained over $\bigotimes_{j=1}^{6}[O_j - 1.5191, O_j + 1.5191]$ and TFFN is trained over $\bigotimes_{j=1}^{6}[O_j - 1.2, O_j + 1.2]$. 
We observe in the table that both TNNs perform similarly to the neural network in the data-driven method. Unlike the data-driven method in \cite{zhai2022DL4FP}, our method does not require using SDE simulations during the training.
}

\subsection{\color{black}Using different optimization methods}
{\color{black}
We test different optimization methods in Example \ref{exm:2d-ring}.

First, we test ADAM and LION in Example \ref{exm:2d-ring}, using TFFN.
We test both ADAM and LION with the same learning rate schedule: starting with a learning rate $10^{-3}$, decaying to $8\times 10^{-6}$ using a polynomial schedule. 
Both are trained with $10^5$ epochs and a batch size $4096$. 
We report the average relative errors in Table \ref{abl-optimizer-TFFN}. 
\begin{table}[!ht]
    \centering
    \begin{tabular}{lccccc} 
        \toprule
        & & \multicolumn{3}{c}{density region/error } \\ 
        \cmidrule{3-5} 
         model &   Optimizer & $\Gamma_{1\times 10^{-2}}$  & $\Gamma_{5\times 10^{-2}}$ & $\Gamma_{1\times 10^{-1}}$ \\ 
        \midrule
          {TFFN}(128, [1 8 8 8 1]) & ADAM   & 0.0033 & 0.0025 & 0.0023  \\ 
         {TFFN}(128, [1 8 8 8 1])  & LION  & 0.0012 &   0.0008  & 0.0007 \\ 
         \midrule
           $\#$  of  test points $n$ & & 44495 & 35392 & 26884    \\
        \bottomrule
    \end{tabular}
    \caption{\small{
    Ablation study: Optimization methods with TFFN for Example \ref{exm:2d-ring}: average relative errors \eqref{eq:defn-relative-error}. All $5 \times 10^5$ test points $x_i$ are sampled uniformly from $\Gamma = [-2, 2]^2$.
    }}
    \label{abl-optimizer-TFFN}
\end{table}

Second, we test TRBFN using four different optimization methods:
LION, ADAM, and SGD, and a two-step method from \cite{JMLR:v12:gonen11a} for multiple kernel learning. 
For LION, ADAM and SGD, we use the same polynomial learning rate schedule: starting from $9 \times 10^{-4}$ to  $8\times 10^{-6}$. The batch size is $5000$ and we train  $10^5$ epochs. 
In the two-step method in \cite{JMLR:v12:gonen11a}, the combination parameters $c_i$, $\alpha_{ij}^{(\ell)}$ in \eqref{eq:KDE Estimator} 
and \eqref{eq:additive-kernel-1d} and the base learner parameters $s_{ij}^{(\ell)}$ and $h_{ij}^{(\ell)}$ in \eqref{eq:additive-kernel-1d} are updated alternately during the iterations. For each update, we use LION, for both combination and base learner parameters. For the two-step method,
we first update combination parameters for $100$ epochs and then update base learner parameters for $100$ epochs and repeat this process until we reach $2 \times 10^5$ epochs. The batch size is $5000$. In each step, we use the same polynomial learning rate schedule decaying from $9\times 10^{-4}$ to $8\times 10^{-6}$.

We test TRBFN(1000, 3)  with 
each $k_{ij}$  being Wendland radial basis functions using 
these four optimization methods. 
We report the results in Table \ref{tab:TRBFN-opt} and plot the evolution of the loss function for LION, ADAM and SGD in Figure \ref{fig:loss-opt}. The results in Figure \ref{fig:loss-opt} demonstrate that LION minimizes the loss most effectively. This suggests that training TRBFN with LION  may achieve the best accuracy.

\begin{table}[!ht]
    \centering
    \begin{tabular}{lccccc} 
        \toprule
        & & & \multicolumn{3}{c}{density region/error } \\ 
        \cmidrule{4-6} 
         model &   Optimizer & L2 error & $\Gamma_{1 \times 10^{-2}}$ & $\Gamma_{5 \times 10^{-2}}$& $\Gamma_{1 \times 10^{-1}}$ \\ 
        \midrule
          {TRBFN}(1000, 3) & LION & $8\times 10^{-6}$ & 0.0001 & 0.0001 & 0.0001  \\ 
         {TRBFN}(1000, 3)  & ADAM & $1.3\times 10^{-2}$ & 0.1140 & 0.1085 & 0.1099 \\ 
         {TRBFN}(1000, 3)  & SGD & $1.3\times 10^{-2}$ & 0.1317 & 0.1157 & 0.1138 \\ 
         {TRBFN}(1000, 3)  & Two-step (with LION) & $2.3\times 10^{-5}$ & 0.0002 & 0.0002 & 0.0002 \\ 
         \midrule
           $\#$  of  test points $n$ & & & 44495 & 35392 & 26884    \\
        \bottomrule
    \end{tabular}
    \caption{\small{
    Ablation study: Optimization methods with TRBFN for Example \ref{exm:2d-ring}: average relative errors \eqref{eq:defn-relative-error}. All $ 10^5$ test points $x_i$ are sampled uniformly on $\Gamma = [-2, 2]^2$.
    }}
    \label{tab:TRBFN-opt}
\end{table}

\begin{figure}
\centering
\includegraphics[width=0.55\linewidth]{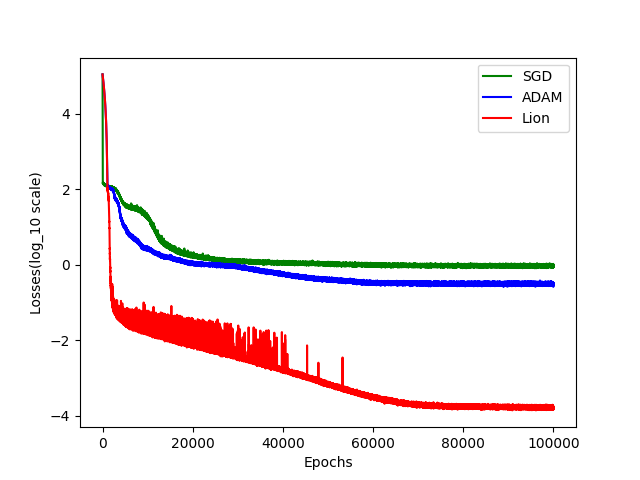}
\caption{Loss Evolution Graph of SGD, ADAM and LION for TRBFN Example \ref{exm:2d-ring}}
\label{fig:loss-opt}
\end{figure}
 }
%
\subsection{\color{black}Enforcing boundary condition or not}
In this subsection, we test the  effect of enforcing boundary conditions for both TFFN and TRBFN, with Example \ref{exm:6d-unimode}.

To enforce the boundary condition for TFFN, we use the vanishing boundary condition with the truncation function \eqref{eqn:trunc}. 
We test TFFN on $\bigotimes_{j=1}^{6} [O_j - 1.2, O_j + 1.2]$ without truncation function for Example \ref{exm:6d-unimode} and we find  $99.99\%$ relative error. This implies the sensitivity in the boundary conditions for TFFN.

For Example \ref{exm:6d-unimode} and bounded domain $\bigotimes_{j=1}^{6} [O_j - 1.5191, O_j + 1.5191 ]$, we test TRBFN(800, 6) which contains
three Wendland radial basis functions (compactly supported) and three inverse quadratic radial basis functions (positive) for each $k_{ij}$. 
In Table \ref{tab:abl-penalty-term}, we report the results from TRBFN with and without the boundary penalty term \eqref{penalty-boundary}:
\begin{equation}\label{penalty-boundary}
    \sum_{i,j}( (k_{ij}(O_j + r)  + k_{ij}(O_j - r) ).
\end{equation}

We observe that TRBFN still works
even without the boundary condition explicitly imposed. We attribute this performance to the specific choices of decaying radial basis functions and the penalization term $ \sum_{i,j,\ell}\left(\max(|s_{ij}^{(\ell)} - O_j| - r, 0 ) + \max(|h_{ij}^{(\ell)}| -  (|r - |s_{ij}^{(\ell)} - O_j|\,|), 0)\right)$.   However,  the accuracy of TRBFN can be improved when the boundary condition is imposed via \eqref{penalty-boundary}.
\begin{table}[!ht]
    \centering
    \begin{tabular}{lccccc} 
        \toprule
        & & & \multicolumn{3}{c}{density region/error } \\ 
        \cmidrule{4-6} 
         model &   boundary condition & L2 error & $\Gamma_{5\times 10^{-2}}$  & $\Gamma_{2.5\times 10^{-1}}$ & $\Gamma_{5\times 10^{-1}}$ \\ 
        \midrule
          {TRBFN}(800, 6)(wi) & with \eqref{penalty-boundary}  & $5.4\times 10^{-3}$ & 0.0974 & 0.0599 & 0.0469  \\ 
         {TRBFN}(800, 6)(wi)  & without 
         \eqref{penalty-boundary}  & $8.4\times 10^{-3}$ & 0.1534 & 0.0947 & 0.068 \\ 
         \midrule
           $\#$  of  test points $n$ & & & 34705 &     7796 &    1926    \\
        \bottomrule
    \end{tabular}
    \caption{\small{
    Ablation study: Penalty term of boundary conditions for  TRBFN with Example \ref{exm:6d-unimode}: average relative errors \eqref{eq:defn-relative-error}. All $5 \times 10^5$ test points $x_i$ are sampled uniformly on $\Gamma = [-1, 1]^6$. 
    }}
    \label{tab:abl-penalty-term}
\end{table}

\subsection{Using different RBFs in TRBFN}

In this section, we investigate how the different combinations of radial basis functions affect the performance of TRBFN.
We compare two cases:
Case 1: $k_{ij}$ only contains compactly-supported radial basis functions; Case 2:  $k_{ij}$ contains compactly supported and positive radial basis functions. 
We use Example \ref{exm:6d-unimode} over the  numerical support $\bigotimes_{j=1}^{10}[O_j - 1.5191, O_j + 1.5191]$.

For Case 1, we choose TRBFN with rank 800, and each $k_{ij}$ contains six Wendland's radial basis functions (compactly-supported). 
For Case 2, we use  TRBFN with rank 800 but each $k_{ij}$ contains three Wendland's radial basis functions and three inverse quadratic radial basis functions (positive). 
In Table \ref{tab:abl-rbf_types}, we report the relative average errors and observe higher accuracy in Case 1. This difference can be explained by the fact that $k_{ij}$'s in TRBFN  only contain compactly supported radial basis functions and are zero outside the numerical support.


\begin{table}[!ht]
    \centering
    \begin{tabular}{lccccc} 
        \toprule
        & & & \multicolumn{3}{c}{density region/error } \\ 
        \cmidrule{4-6} 
         model &   RBF types & L2 error & $\Gamma_{5\times 10^{-2}}$  & $\Gamma_{2.5\times 10^{-1}}$ & $\Gamma_{5\times 10^{-1}}$ \\ 
        \midrule
          {TRBFN}(800, 6)  & Case 1 & $2.5\times 10^{-3}$ & 0.0451 & 0.0271 & 0.0216  \\ 
         {TRBFN}(800, 6)  & Case 2 & $5.4\times 10^{-3}$ & 0.0974 & 0.0599 & 0.0469 \\ 
         \midrule
           $\#$  of  test points $n$ & & & 34705 &     7796 &    1926    \\
        \bottomrule
    \end{tabular}
    \caption{\small{
    Ablation study: RBF types of TRBFN with Example \ref{exm:6d-unimode}: average relative errors \eqref{eq:defn-relative-error} for different tensor neural networks. All $5 \times 10^5$ test points $x_i$ are sampled uniformly from $\Gamma = [-1, 1]^6$.
    }}
    \label{tab:abl-rbf_types}
\end{table}

\subsection{\zz{Using different arithmetic precision}}
\rone{
In this section, we explore the effects of different precision. We test TRBFN(5, 720) and TFFN(64,[1 8 8 1]) with Example \ref{exm:2d-ring} on bounded domain $\bigotimes_{j=1}^2[O - 2.1467, O + 2.1467]$ using single and double precision.

\begin{table}[!ht]
    \centering
    \begin{tabular}{lccccc} 
        \toprule
        & & & \multicolumn{3}{c}{density region/error } \\ 
        \cmidrule{4-6} 
         model &   Precision & L2 error & $\Gamma_{1 \times 10^{-2}}$ & $\Gamma_{5 \times 10^{-2}}$& $\Gamma_{1 \times 10^{-1}}$ \\ 
        \midrule
          TFFN(64,[1 8 8 1]) & float 32  & $4.6\times 10^{-4}$ & 0.0044 & 0.0028 & 0.0022  \\ 
         TFFN(64,[1 8 8 1])  & float 64 & $4.0\times 10^{-4}$ & 0.0030 & 0.0021 & 0.0018  \\ 
         TRBFN(5, 720)  & float 32 & $3.4\times 10^{-4}$ & 0.0051 & 0.0029 & 0.0022 \\ 
         TRBFN(5, 720)  & float 64 & $3.4\times 10^{-4}$ & 0.0051 & 0.0029 & 0.0022 \\ 
         \midrule
           $\#$  of  test points $n$ & & & 89164    &   76013    &   54920    \\
        \bottomrule
    \end{tabular}
    \caption{\small{
    Ablation study: Using single and double precison for TRBFN with Example \ref{exm:2d-ring}: average relative errors \eqref{eq:defn-relative-error} for different tensor neural networks. All $ 10^5$ test points $x_i$ are sampled uniformly from $\Gamma = [-2, 2]^2$.
    }}
    \label{tab:abl-precision}
\end{table}

For TRBFN, the number of epochs is $10^5$ and batch size is $5000$. For TFFN, the number of epochs is $10^5$ and the batch size is $4096$. 
We keep all the hyperparameters and training strategies as in Example \ref{exm:2d-ring} and 
use two precision: single (float 32) and double (float 64). 
From the results in Table \ref{tab:abl-precision}, we observe that the relative error of TFFN from double precision is smaller than that of TFFN with single precision. 
In contrast, the TRBFN achieves the same relative error under both precisions.
The difference of relative error in TFFN may be
caused by the numerical integration error when  $Z(\Theta)$  in  Equation \eqref{eq:estimator-integral} is computed while   $Z(\Theta)$ is computed analytically when TRBFN is used.
}

%
\section{Summary and discussion}\label{sec:summary}
We have investigated tensor  radial basis function networks and tensor-feedforward neural networks for  steady-state Fokker-Plank equations in high dimensions. 
We first designed a simple procedure for estimating efficient numerical supports and then applied these two types of networks for 
the equations over the numerical supports. 
For tensor radial basis functions, 
we imposed conditions to control the parameters in each radial basis function for better approximations.  
The radial basis functions allowed us to find the exact derivatives and integration over numerical support and thus significantly reduce storage requirements from auto-differentiation in major Python packages. 
%
The tensor feedforward networks are straightforward to implement but are more sensitive to numerical supports even in six or less dimensions.  
We presented several examples  \rone{as well as various ablation studies}
to illustrate the feasibility and accuracy of both tensor networks.

There are also some limitations of the current methodology. First, we use a  hypercube as numerical support while for Fokker-Planck equations the numerical support may be large when the density function decay slowly or contains separated regions.  Also, the efficiency is sensitive to the choice of numerical support as it may require more training points in a  large support. Second, the choice of the kernels is somewhat arbitrary.
More kernel functions of different kind may be employed for larger approximation capacity; see e.g.  \cite{BouOwh21} and consequent works.   Third, we imposed some constraints on the parameters in radial basis functions, which are general but still are problem-dependent. An alternative approach is to use 
an alternating direction method in training these networks and apply a cross-validation strategy to determine when to stop computation, such as in \cite{Akian22learningbestkernel,owhadi2019kernel,yang2024learning}. \rone{Fourth, the approximation error estimates of the tensor networks and efficiency of the tensor neural networks require theoretical study.} We leave these issues for future study.

The methodology in this work may be combined with DeepBSDE or other methods for better accuracy, where an accurate solution is desired at a few points in the high-probability regions. For example, we may run tensor neural networks to obtain an estimate of high probability region and then compute the solution at the points of interest with DeepBSDE. 

\vskip 10pt 
\section*{Declaration of competing interest}

The authors declare that they have no known competing financial interests or personal relationships that could have
appeared to influence the work reported in this paper.

\vskip 10pt 
\section*{Data availability}

The code and data generated during and/or analysed during the current study will be available upon publication.

\section*{Acknowledgement}
We would like to thank Professor Houman Ohwadi for helpful discussions and for providing valuable references. 
ZZ was partially supported by AFOSR under award number FA9550-20-1-0056. 
GK acknowledges support by the DOE SEA-CROGS project (DE-SC0023191) and the MURI-AFOSR project (FA9550-20-1-0358), and the ONR Vannevar Bush Faculty Fellowship (N00014-22-1-2795). 
This research of ZH and KK is partially supported by the National Research Foundation Singapore under the AI Singapore Programme (AISG Award No: AISG2-TC-2023-010-SGIL) and the  Tier 1 (Award No: T1 251RES2207).
\zz{
The authors are grateful for the reviewers'  comments and suggestions, which helped us in improving the quality of the manuscript. }
\bibliographystyle{elsarticle-num-names}
\bibliography{density,kernel,highdpinn,nntraining,fokkerplanck}

\begin{thebibliography}{70}
\expandafter\ifx\csname natexlab\endcsname\relax\def\natexlab#1{#1}\fi
\providecommand{\url}[1]{\texttt{#1}}
\providecommand{\href}[2]{#2}
\providecommand{\path}[1]{#1}
\providecommand{\DOIprefix}{doi:}
\providecommand{\ArXivprefix}{arXiv:}
\providecommand{\URLprefix}{URL: }
\providecommand{\Pubmedprefix}{pmid:}
\providecommand{\doi}[1]{\href{http://dx.doi.org/#1}{\path{#1}}}
\providecommand{\Pubmed}[1]{\href{pmid:#1}{\path{#1}}}
\providecommand{\bibinfo}[2]{#2}
\ifx\xfnm\relax \def\xfnm[#1]{\unskip,\space#1}\fi
\bibitem[{Risken(1996)}]{Risken96-book}
\bibinfo{author}{H.~Risken}, \bibinfo{title}{The {F}okker-{P}lanck equation},
  volume~\bibinfo{volume}{18} of \textit{\bibinfo{series}{Springer Series in
  Synergetics}}, \bibinfo{edition}{2nd} ed.,
  \bibinfo{publisher}{Springer-Verlag, Berlin}, \bibinfo{year}{1996}.
\bibitem[{Raissi et~al.(2019)Raissi, Perdikaris, and
  Karniadakis}]{raissi2019physics}
\bibinfo{author}{M.~Raissi}, \bibinfo{author}{P.~Perdikaris},
  \bibinfo{author}{G.~E. Karniadakis},
\newblock \bibinfo{title}{Physics-informed neural networks: A deep learning
  framework for solving forward and inverse problems involving nonlinear
  partial differential equations},
\newblock \bibinfo{journal}{Journal of Computational Physics}
  \bibinfo{volume}{378} (\bibinfo{year}{2019}) \bibinfo{pages}{686--707}.
\bibitem[{Jordan et~al.(1998)Jordan, Kinderlehrer, and
  Otto}]{JordanKD1998variationalFP}
\bibinfo{author}{R.~Jordan}, \bibinfo{author}{D.~Kinderlehrer},
  \bibinfo{author}{F.~Otto},
\newblock \bibinfo{title}{The variational formulation of the {F}okker-{P}lanck
  equation},
\newblock \bibinfo{journal}{SIAM J. Math. Anal.} \bibinfo{volume}{29}
  (\bibinfo{year}{1998}) \bibinfo{pages}{1--17}. \URLprefix
  \url{https://doi.org/10.1137/S0036141096303359}.
  \DOIprefix\doi{10.1137/S0036141096303359}.
\bibitem[{Liu et~al.(2022)Liu, Li, Zha, and Zhou}]{Liu22NeuralParametricFP}
\bibinfo{author}{S.~Liu}, \bibinfo{author}{W.~Li}, \bibinfo{author}{H.~Zha},
  \bibinfo{author}{H.~Zhou},
\newblock \bibinfo{title}{Neural parametric fokker--planck equation},
\newblock \bibinfo{journal}{SIAM Journal on Numerical Analysis}
  \bibinfo{volume}{60} (\bibinfo{year}{2022}) \bibinfo{pages}{1385--1449}.
  \URLprefix \url{https://doi.org/10.1137/20M1344986}.
  \DOIprefix\doi{10.1137/20M1344986}.
\bibitem[{Kingma and Ba(2015)}]{kingma2014adam}
\bibinfo{author}{D.~P. Kingma}, \bibinfo{author}{J.~Ba},
\newblock \bibinfo{title}{Adam: A method for stochastic optimization},
\newblock \bibinfo{journal}{ICLR}  (\bibinfo{year}{2015}).
\bibitem[{E et~al.(2021)E, Han, and Jentzen}]{EHanJ22review}
\bibinfo{author}{W.~E}, \bibinfo{author}{J.~Han}, \bibinfo{author}{A.~Jentzen},
\newblock \bibinfo{title}{Algorithms for solving high dimensional pdes: from
  nonlinear monte carlo to machine learning},
\newblock \bibinfo{journal}{Nonlinearity} \bibinfo{volume}{35}
  (\bibinfo{year}{2021}) \bibinfo{pages}{278}. \URLprefix
  \url{https://dx.doi.org/10.1088/1361-6544/ac337f}.
  \DOIprefix\doi{10.1088/1361-6544/ac337f}.
\bibitem[{Karniadakis et~al.(2021)Karniadakis, Kevrekidis, Lu, Perdikaris,
  Wang, and Yang}]{karniadakis2021physics}
\bibinfo{author}{G.~E. Karniadakis}, \bibinfo{author}{I.~G. Kevrekidis},
  \bibinfo{author}{L.~Lu}, \bibinfo{author}{P.~Perdikaris},
  \bibinfo{author}{S.~Wang}, \bibinfo{author}{L.~Yang},
\newblock \bibinfo{title}{Physics-informed machine learning},
\newblock \bibinfo{journal}{Nature Reviews Physics} \bibinfo{volume}{3}
  (\bibinfo{year}{2021}) \bibinfo{pages}{422--440}.
\bibitem[{Genz(1984)}]{genz84}
\bibinfo{author}{A.~Genz},
\newblock \bibinfo{title}{Testing multidimensional integration routines},
\newblock in: \bibinfo{booktitle}{Proceedings of the International Conference
  on Tools, Methods and Languages for Scientific and Engineering Computation},
  \bibinfo{publisher}{Elsevier North-Holland, Amsterdam}, \bibinfo{year}{1984},
  pp. \bibinfo{pages}{81--94}.
\bibitem[{Alhussein et~al.(2023)Alhussein, Khasawneh, and
  Daqaq}]{alhussein2023physics}
\bibinfo{author}{H.~Alhussein}, \bibinfo{author}{M.~Khasawneh},
  \bibinfo{author}{M.~F. Daqaq},
\newblock \bibinfo{title}{Physics-informed solution of the stationary
  fokker-plank equation for a class of nonlinear dynamical systems: An
  evaluation study},
\newblock \bibinfo{journal}{arXiv preprint arXiv:2309.16725}
  (\bibinfo{year}{2023}).
\bibitem[{Dobson et~al.(2022)Dobson, Li, and Zhai}]{Dobson22-fokkerplanck}
\bibinfo{author}{M.~Dobson}, \bibinfo{author}{Y.~Li},
  \bibinfo{author}{J.~Zhai},
\newblock \bibinfo{title}{An efficient data-driven solver for {F}okker-{P}lanck
  equations: algorithm and analysis},
\newblock \bibinfo{journal}{Commun. Math. Sci.} \bibinfo{volume}{20}
  (\bibinfo{year}{2022}) \bibinfo{pages}{803--827}. \URLprefix
  \url{https://doi.org/10.4310/CMS.2022.v20.n3.a8}.
  \DOIprefix\doi{10.4310/CMS.2022.v20.n3.a8}.
\bibitem[{Zhai et~al.(2022)Zhai, Dobson, and Li}]{zhai2022DL4FP}
\bibinfo{author}{J.~Zhai}, \bibinfo{author}{M.~Dobson},
  \bibinfo{author}{Y.~Li},
\newblock \bibinfo{title}{A deep learning method for solving {F}okker-{P}lanck
  equations},
\newblock in: \bibinfo{editor}{J.~Bruna}, \bibinfo{editor}{J.~Hesthaven},
  \bibinfo{editor}{L.~Zdeborova} (Eds.), \bibinfo{booktitle}{Proceedings of the
  2nd Mathematical and Scientific Machine Learning Conference}, volume
  \bibinfo{volume}{145} of \textit{\bibinfo{series}{Proceedings of Machine
  Learning Research}}, \bibinfo{publisher}{PMLR}, \bibinfo{year}{2022}, pp.
  \bibinfo{pages}{568--597}. \URLprefix
  \url{https://proceedings.mlr.press/v145/zhai22a.html}.
\bibitem[{Cho et~al.(2022)Cho, Nam, Yang, Yun, Hong, and
  Park}]{cho2022separable}
\bibinfo{author}{J.~Cho}, \bibinfo{author}{S.~Nam}, \bibinfo{author}{H.~Yang},
  \bibinfo{author}{S.-B. Yun}, \bibinfo{author}{Y.~Hong},
  \bibinfo{author}{E.~Park},
\newblock \bibinfo{title}{Separable pinn: Mitigating the curse of
  dimensionality in physics-informed neural networks},
\newblock \bibinfo{journal}{arXiv:2211.08761}  (\bibinfo{year}{2022}).
\bibitem[{Wang et~al.(2022{\natexlab{a}})Wang, Jin, and Xie}]{wang2022tensor}
\bibinfo{author}{Y.~Wang}, \bibinfo{author}{P.~Jin}, \bibinfo{author}{H.~Xie},
\newblock \bibinfo{title}{Tensor neural network and its numerical integration},
\newblock \bibinfo{journal}{arXiv:2207.02754}
  (\bibinfo{year}{2022}{\natexlab{a}}).
\bibitem[{Wang et~al.(2022{\natexlab{b}})Wang, Liao, and Xie}]{wang2022solving}
\bibinfo{author}{Y.~Wang}, \bibinfo{author}{Y.~Liao}, \bibinfo{author}{H.~Xie},
\newblock \bibinfo{title}{Solving schr$\backslash$"$\{$o$\}$ dinger equation
  using tensor neural network},
\newblock \bibinfo{journal}{arXiv:2209.12572}
  (\bibinfo{year}{2022}{\natexlab{b}}).
\bibitem[{Zhang et~al.(2023)Zhang, Xu, Liu, and Li}]{zhang2023low-rank}
\bibinfo{author}{H.~Zhang}, \bibinfo{author}{Y.~Xu}, \bibinfo{author}{Q.~Liu},
  \bibinfo{author}{Y.~Li},
\newblock \bibinfo{title}{Deep learning framework for solving {F}okker-{P}lanck
  equations with low-rank separation representation},
\newblock \bibinfo{journal}{Engineering Applications of Artificial
  Intelligence} \bibinfo{volume}{121} (\bibinfo{year}{2023})
  \bibinfo{pages}{106036}.
\bibitem[{Han et~al.(2018)Han, Jentzen, and E}]{han2018solving}
\bibinfo{author}{J.~Han}, \bibinfo{author}{A.~Jentzen}, \bibinfo{author}{W.~E},
\newblock \bibinfo{title}{Solving high-dimensional partial differential
  equations using deep learning},
\newblock \bibinfo{journal}{Proceedings of the National Academy of Sciences}
  \bibinfo{volume}{115} (\bibinfo{year}{2018}) \bibinfo{pages}{8505--8510}.
\bibitem[{Han et~al.(2017)Han, Jentzen et~al.}]{han2017deep}
\bibinfo{author}{J.~Han}, \bibinfo{author}{A.~Jentzen}, et~al.,
\newblock \bibinfo{title}{Deep learning-based numerical methods for
  high-dimensional parabolic partial differential equations and backward
  stochastic differential equations},
\newblock \bibinfo{journal}{Communications in mathematics and statistics}
  \bibinfo{volume}{5} (\bibinfo{year}{2017}) \bibinfo{pages}{349--380}.
\bibitem[{Beck et~al.(2019)Beck, E, and Jentzen}]{beck2019machine}
\bibinfo{author}{C.~Beck}, \bibinfo{author}{W.~E},
  \bibinfo{author}{A.~Jentzen},
\newblock \bibinfo{title}{Machine learning approximation algorithms for
  high-dimensional fully nonlinear partial differential equations and
  second-order backward stochastic differential equations},
\newblock \bibinfo{journal}{Journal of Nonlinear Science} \bibinfo{volume}{29}
  (\bibinfo{year}{2019}) \bibinfo{pages}{1563--1619}.
\bibitem[{Beck et~al.(2020)Beck, Hornung, Hutzenthaler, Jentzen, and
  Kruse}]{beck2020overcoming_ac}
\bibinfo{author}{C.~Beck}, \bibinfo{author}{F.~Hornung},
  \bibinfo{author}{M.~Hutzenthaler}, \bibinfo{author}{A.~Jentzen},
  \bibinfo{author}{T.~Kruse},
\newblock \bibinfo{title}{Overcoming the curse of dimensionality in the
  numerical approximation of {A}llen--{C}ahn partial differential equations via
  truncated full-history recursive multilevel {P}icard approximations},
\newblock \bibinfo{journal}{Journal of Numerical Mathematics}
  \bibinfo{volume}{28} (\bibinfo{year}{2020}) \bibinfo{pages}{197--222}.
\bibitem[{Beck et~al.(2021)Beck, Becker, Cheridito, Jentzen, and
  Neufeld}]{beck2021deep}
\bibinfo{author}{C.~Beck}, \bibinfo{author}{S.~Becker},
  \bibinfo{author}{P.~Cheridito}, \bibinfo{author}{A.~Jentzen},
  \bibinfo{author}{A.~Neufeld},
\newblock \bibinfo{title}{Deep splitting method for parabolic {PDE}s},
\newblock \bibinfo{journal}{SIAM Journal on Scientific Computing}
  \bibinfo{volume}{43} (\bibinfo{year}{2021}) \bibinfo{pages}{A3135--A3154}.
\bibitem[{Becker et~al.(2021)Becker, Cheridito, Jentzen, and
  Welti}]{becker2021solving}
\bibinfo{author}{S.~Becker}, \bibinfo{author}{P.~Cheridito},
  \bibinfo{author}{A.~Jentzen}, \bibinfo{author}{T.~Welti},
\newblock \bibinfo{title}{Solving high-dimensional optimal stopping problems
  using deep learning},
\newblock \bibinfo{journal}{European Journal of Applied Mathematics}
  \bibinfo{volume}{32} (\bibinfo{year}{2021}) \bibinfo{pages}{470--514}.
\bibitem[{Chan-Wai-Nam et~al.(2019)Chan-Wai-Nam, Mikael, and
  Warin}]{chan2019machine}
\bibinfo{author}{Q.~Chan-Wai-Nam}, \bibinfo{author}{J.~Mikael},
  \bibinfo{author}{X.~Warin},
\newblock \bibinfo{title}{Machine learning for semi-linear {PDE}s},
\newblock \bibinfo{journal}{Journal of scientific computing}
  \bibinfo{volume}{79} (\bibinfo{year}{2019}) \bibinfo{pages}{1667--1712}.
\bibitem[{Henry-Labordere(2017)}]{henry2017deep}
\bibinfo{author}{P.~Henry-Labordere},
\newblock \bibinfo{title}{Deep primal-dual algorithm for bsdes: Applications of
  machine learning to cva and im},
\newblock \bibinfo{journal}{Available at SSRN 3071506}  (\bibinfo{year}{2017}).
\bibitem[{Hur{\'e} et~al.(2020)Hur{\'e}, Pham, and Warin}]{hure2020deep}
\bibinfo{author}{C.~Hur{\'e}}, \bibinfo{author}{H.~Pham},
  \bibinfo{author}{X.~Warin},
\newblock \bibinfo{title}{Deep backward schemes for high-dimensional nonlinear
  {PDE}s},
\newblock \bibinfo{journal}{Mathematics of Computation} \bibinfo{volume}{89}
  (\bibinfo{year}{2020}) \bibinfo{pages}{1547--1579}.
\bibitem[{Hutzenthaler et~al.(2020)Hutzenthaler, Jentzen, Kruse, Anh~Nguyen,
  and von Wurstemberger}]{hutzenthaler2020overcoming}
\bibinfo{author}{M.~Hutzenthaler}, \bibinfo{author}{A.~Jentzen},
  \bibinfo{author}{T.~Kruse}, \bibinfo{author}{T.~Anh~Nguyen},
  \bibinfo{author}{P.~von Wurstemberger},
\newblock \bibinfo{title}{Overcoming the curse of dimensionality in the
  numerical approximation of semilinear parabolic partial differential
  equations},
\newblock \bibinfo{journal}{Proceedings of the Royal Society A}
  \bibinfo{volume}{476} (\bibinfo{year}{2020}) \bibinfo{pages}{20190630}.
\bibitem[{Ji et~al.(2020)Ji, Peng, Peng, and Zhang}]{ji2020three}
\bibinfo{author}{S.~Ji}, \bibinfo{author}{S.~Peng}, \bibinfo{author}{Y.~Peng},
  \bibinfo{author}{X.~Zhang},
\newblock \bibinfo{title}{Three algorithms for solving high-dimensional fully
  coupled {FBSDE}s through deep learning},
\newblock \bibinfo{journal}{IEEE Intelligent Systems} \bibinfo{volume}{35}
  (\bibinfo{year}{2020}) \bibinfo{pages}{71--84}.
\bibitem[{Raissi(2018)}]{raissi2018forward}
\bibinfo{author}{M.~Raissi},
\newblock \bibinfo{title}{Forward-backward stochastic neural networks: Deep
  learning of high-dimensional partial differential equations},
\newblock \bibinfo{journal}{arXiv:1804.07010}  (\bibinfo{year}{2018}).
\bibitem[{Zhang et~al.(2022)Zhang, Zhang, and
  Zhou}]{zhang2022predictorcorrector}
\bibinfo{author}{H.~Zhang}, \bibinfo{author}{R.~Zhang},
  \bibinfo{author}{T.~Zhou},
\newblock \bibinfo{title}{A predictor-corrector deep learning algorithm for
  high dimensional stochastic partial differential equations},
\newblock \bibinfo{journal}{arXiv:2208.09883}  (\bibinfo{year}{2022}).
\bibitem[{He et~al.(2023)He, Li, Shi, Gao, Zhang, Bian, Wang, and
  Liu}]{he2023learning}
\bibinfo{author}{D.~He}, \bibinfo{author}{S.~Li}, \bibinfo{author}{W.~Shi},
  \bibinfo{author}{X.~Gao}, \bibinfo{author}{J.~Zhang},
  \bibinfo{author}{J.~Bian}, \bibinfo{author}{L.~Wang}, \bibinfo{author}{T.-Y.
  Liu},
\newblock \bibinfo{title}{Learning physics-informed neural networks without
  stacked back-propagation},
\newblock in: \bibinfo{booktitle}{International Conference on Artificial
  Intelligence and Statistics}, \bibinfo{organization}{PMLR},
  \bibinfo{year}{2023}, pp. \bibinfo{pages}{3034--3047}.
\bibitem[{Hu et~al.(2023)Hu, Shukla, Karniadakis, and
  Kawaguchi}]{hu2023tackling}
\bibinfo{author}{Z.~Hu}, \bibinfo{author}{K.~Shukla}, \bibinfo{author}{G.~E.
  Karniadakis}, \bibinfo{author}{K.~Kawaguchi},
\newblock \bibinfo{title}{Tackling the curse of dimensionality with
  physics-informed neural networks},
\newblock \bibinfo{journal}{arXiv:2307.12306}  (\bibinfo{year}{2023}).
\bibitem[{Mishra et~al.(2018)Mishra, Nath, Sen, and
  Fasshauer}]{mishra2018hybrid}
\bibinfo{author}{P.~K. Mishra}, \bibinfo{author}{S.~K. Nath},
  \bibinfo{author}{M.~K. Sen}, \bibinfo{author}{G.~E. Fasshauer},
\newblock \bibinfo{title}{Hybrid {G}aussian-cubic radial basis functions for
  scattered data interpolation},
\newblock \bibinfo{journal}{Comput. Geosci.} \bibinfo{volume}{22}
  (\bibinfo{year}{2018}) \bibinfo{pages}{1203--1218}.
  \DOIprefix\doi{10.1007/s10596-018-9747-3}.
\bibitem[{Mishra et~al.(2019)Mishra, Fasshauer, Sen, and Ling}]{MisFS19}
\bibinfo{author}{P.~K. Mishra}, \bibinfo{author}{G.~E. Fasshauer},
  \bibinfo{author}{M.~K. Sen}, \bibinfo{author}{L.~Ling},
\newblock \bibinfo{title}{A stabilized radial basis-finite difference
  ({RBF}-{FD}) method with hybrid kernels},
\newblock \bibinfo{journal}{Comput. Math. Appl.} \bibinfo{volume}{77}
  (\bibinfo{year}{2019}) \bibinfo{pages}{2354--2368}. \URLprefix
  \url{https://doi.org/10.1016/j.camwa.2018.12.027}.
  \DOIprefix\doi{10.1016/j.camwa.2018.12.027}.
\bibitem[{Senel et~al.(2022)Senel, van Beeck, and Altinkaynak}]{SenvA22}
\bibinfo{author}{C.~B. Senel}, \bibinfo{author}{J.~van Beeck},
  \bibinfo{author}{A.~Altinkaynak},
\newblock \bibinfo{title}{Solving {PDE}s with a hybrid radial basis function:
  power-generalized multiquadric kernel},
\newblock \bibinfo{journal}{Adv. Appl. Math. Mech.} \bibinfo{volume}{14}
  (\bibinfo{year}{2022}) \bibinfo{pages}{1161--1180}.
  \DOIprefix\doi{10.4208/aamm.oa-2021-0215}.
\bibitem[{Baek and Kim(2019)}]{baek_new_2019}
\bibinfo{author}{J.~Baek}, \bibinfo{author}{E.~Kim},
\newblock \bibinfo{title}{A new support vector machine with an optimal additive
  kernel},
\newblock \bibinfo{journal}{Neurocomputing} \bibinfo{volume}{329}
  (\bibinfo{year}{2019}) \bibinfo{pages}{279--299}. \URLprefix
  \url{https://www.sciencedirect.com/science/article/pii/S0925231218312207}.
  \DOIprefix\doi{10.1016/j.neucom.2018.10.032}.
\bibitem[{Aiolli and Donini(2015)}]{AIOLLI2015215}
\bibinfo{author}{F.~Aiolli}, \bibinfo{author}{M.~Donini},
\newblock \bibinfo{title}{{EasyMKL}: a scalable multiple kernel learning
  algorithm},
\newblock \bibinfo{journal}{Neurocomputing} \bibinfo{volume}{169}
  (\bibinfo{year}{2015}) \bibinfo{pages}{215--224}. \bibinfo{note}{Learning for
  Visual Semantic Understanding in Big Data ESANN 2014 Industrial Data
  Processing and Analysis}.
\bibitem[{Jain et~al.(2012)Jain, Vishwanathan, and Varma}]{jain2012spf}
\bibinfo{author}{A.~Jain}, \bibinfo{author}{S.~V. Vishwanathan},
  \bibinfo{author}{M.~Varma},
\newblock \bibinfo{title}{{SPF-GMKL}: generalized multiple kernel learning with
  a million kernels},
\newblock in: \bibinfo{booktitle}{Proceedings of the 18th ACM SIGKDD
  international conference on Knowledge discovery and data mining},
  \bibinfo{year}{2012}, pp. \bibinfo{pages}{750--758}.
\bibitem[{Kloft et~al.(2011)Kloft, Brefeld, Sonnenburg, and Zien}]{kloft2011lp}
\bibinfo{author}{M.~Kloft}, \bibinfo{author}{U.~Brefeld},
  \bibinfo{author}{S.~Sonnenburg}, \bibinfo{author}{A.~Zien},
\newblock \bibinfo{title}{Lp-norm multiple kernel learning},
\newblock \bibinfo{journal}{The Journal of Machine Learning Research}
  \bibinfo{volume}{12} (\bibinfo{year}{2011}) \bibinfo{pages}{953--997}.
\bibitem[{Moeller et~al.(2014)Moeller, Raman, Venkatasubramanian, and
  Saha}]{moeller2014geometric}
\bibinfo{author}{J.~Moeller}, \bibinfo{author}{P.~Raman},
  \bibinfo{author}{S.~Venkatasubramanian}, \bibinfo{author}{A.~Saha},
\newblock \bibinfo{title}{A geometric algorithm for scalable multiple kernel
  learning},
\newblock in: \bibinfo{booktitle}{Artificial Intelligence and Statistics},
  \bibinfo{organization}{PMLR}, \bibinfo{year}{2014}, pp.
  \bibinfo{pages}{633--642}.
\bibitem[{Orabona and Luo(2011)}]{orabona2011ultra}
\bibinfo{author}{F.~Orabona}, \bibinfo{author}{J.~Luo},
\newblock \bibinfo{title}{Ultra-fast optimization algorithm for sparse multi
  kernel learning},
\newblock in: \bibinfo{booktitle}{Proceedings of the 28th International
  Conference on Machine Learning}, \bibinfo{number}{CONF},
  \bibinfo{year}{2011}.
\bibitem[{Rakotomamonjy et~al.(2008)Rakotomamonjy, Bach, Canu, and
  Grandvalet}]{JMLR:v9:rakotomamonjy08a}
\bibinfo{author}{A.~Rakotomamonjy}, \bibinfo{author}{F.~R. Bach},
  \bibinfo{author}{S.~Canu}, \bibinfo{author}{Y.~Grandvalet},
\newblock \bibinfo{title}{{SimpleMKL}},
\newblock \bibinfo{journal}{Journal of Machine Learning Research}
  \bibinfo{volume}{9} (\bibinfo{year}{2008}) \bibinfo{pages}{2491--2521}.
\bibitem[{Sonnenburg et~al.(2006)Sonnenburg, R{\"a}tsch, Sch{\"a}fer, and
  Sch{\"o}lkopf}]{sonnenburg2006large}
\bibinfo{author}{S.~Sonnenburg}, \bibinfo{author}{G.~R{\"a}tsch},
  \bibinfo{author}{C.~Sch{\"a}fer}, \bibinfo{author}{B.~Sch{\"o}lkopf},
\newblock \bibinfo{title}{Large scale multiple kernel learning},
\newblock \bibinfo{journal}{The Journal of Machine Learning Research}
  \bibinfo{volume}{7} (\bibinfo{year}{2006}) \bibinfo{pages}{1531--1565}.
\bibitem[{Varma and Babu(2009)}]{varma2009more}
\bibinfo{author}{M.~Varma}, \bibinfo{author}{B.~R. Babu},
\newblock \bibinfo{title}{More generality in efficient multiple kernel
  learning},
\newblock in: \bibinfo{booktitle}{Proceedings of the 26th annual international
  conference on machine learning}, \bibinfo{year}{2009}, pp.
  \bibinfo{pages}{1065--1072}.
\bibitem[{Owhadi and Yoo(2019)}]{OwhYoo19}
\bibinfo{author}{H.~Owhadi}, \bibinfo{author}{G.~R. Yoo},
\newblock \bibinfo{title}{Kernel flows: from learning kernels from data into
  the abyss},
\newblock \bibinfo{journal}{J. Comput. Phys.} \bibinfo{volume}{389}
  (\bibinfo{year}{2019}) \bibinfo{pages}{22--47}.
  \DOIprefix\doi{10.1016/j.jcp.2019.03.040}.
\bibitem[{Hamzi and Owhadi(2021)}]{BouOwh21}
\bibinfo{author}{B.~Hamzi}, \bibinfo{author}{H.~Owhadi},
\newblock \bibinfo{title}{Learning dynamical systems from data: a simple
  cross-validation perspective, part {I}: {P}arametric kernel flows},
\newblock \bibinfo{journal}{Phys. D} \bibinfo{volume}{421}
  (\bibinfo{year}{2021}) \bibinfo{pages}{Paper No. 132817, 10}.
  \DOIprefix\doi{10.1016/j.physd.2020.132817}.
\bibitem[{Chen et~al.(2021)Chen, Owhadi, and Stuart}]{ChenOS21}
\bibinfo{author}{Y.~Chen}, \bibinfo{author}{H.~Owhadi}, \bibinfo{author}{A.~M.
  Stuart},
\newblock \bibinfo{title}{Consistency of empirical {B}ayes and kernel flow for
  hierarchical parameter estimation},
\newblock \bibinfo{journal}{Math. Comp.} \bibinfo{volume}{90}
  (\bibinfo{year}{2021}) \bibinfo{pages}{2527--2578}.
  \DOIprefix\doi{10.1090/mcom/3649}.
\bibitem[{Yoo and Owhadi(2021)}]{YooOwh20}
\bibinfo{author}{G.~R. Yoo}, \bibinfo{author}{H.~Owhadi},
\newblock \bibinfo{title}{Deep regularization and direct training of the inner
  layers of neural networks with kernel flows},
\newblock \bibinfo{journal}{Phys. D} \bibinfo{volume}{426}
  (\bibinfo{year}{2021}) \bibinfo{pages}{Paper No. 132952, 7}.
  \DOIprefix\doi{10.1016/j.physd.2021.132952}.
\bibitem[{G{{\"o}}nen and Alpaydin(2011)}]{JMLR:v12:gonen11a}
\bibinfo{author}{M.~G{{\"o}}nen}, \bibinfo{author}{E.~Alpaydin},
\newblock \bibinfo{title}{Multiple kernel learning algorithms},
\newblock \bibinfo{journal}{Journal of Machine Learning Research}
  \bibinfo{volume}{12} (\bibinfo{year}{2011}) \bibinfo{pages}{2211--2268}.
\bibitem[{Chen et~al.(2024)Chen, Liang, Huang, Real, Wang, Pham, Dong, Luong,
  Hsieh, Lu et~al.}]{chen2024symbolic}
\bibinfo{author}{X.~Chen}, \bibinfo{author}{C.~Liang},
  \bibinfo{author}{D.~Huang}, \bibinfo{author}{E.~Real},
  \bibinfo{author}{K.~Wang}, \bibinfo{author}{H.~Pham},
  \bibinfo{author}{X.~Dong}, \bibinfo{author}{T.~Luong}, \bibinfo{author}{C.-J.
  Hsieh}, \bibinfo{author}{Y.~Lu}, et~al.,
\newblock \bibinfo{title}{Symbolic discovery of optimization algorithms},
\newblock \bibinfo{journal}{Advances in neural information processing systems}
  \bibinfo{volume}{36} (\bibinfo{year}{2024}).
\bibitem[{Boffi and Vanden-Eijnden(2023)}]{boffi2023probability}
\bibinfo{author}{N.~M. Boffi}, \bibinfo{author}{E.~Vanden-Eijnden},
\newblock \bibinfo{title}{Probability flow solution of the {F}okker-{P}lanck
  equation},
\newblock \bibinfo{journal}{Machine Learning: Science and Technology}
  \bibinfo{volume}{4} (\bibinfo{year}{2023}) \bibinfo{pages}{035012}.
\bibitem[{Tang et~al.(2022)Tang, Wan, and Liao}]{tang2022adaptive}
\bibinfo{author}{K.~Tang}, \bibinfo{author}{X.~Wan}, \bibinfo{author}{Q.~Liao},
\newblock \bibinfo{title}{Adaptive deep density approximation for
  {F}okker-{P}lanck equations},
\newblock \bibinfo{journal}{J. Comput. Phys.} \bibinfo{volume}{457}
  (\bibinfo{year}{2022}) \bibinfo{pages}{Paper No. 111080, 19}. \URLprefix
  \url{https://doi.org/10.1016/j.jcp.2022.111080}.
  \DOIprefix\doi{10.1016/j.jcp.2022.111080}.
\bibitem[{Anderson and Farazmand(2023)}]{anderson2023fisher}
\bibinfo{author}{W.~Anderson}, \bibinfo{author}{M.~Farazmand},
\newblock \bibinfo{title}{Fisher information and shape-morphing modes for
  solving the {F}okker-{P}lanck equation in higher dimensions},
\newblock \bibinfo{journal}{arXiv: 2306.03749}  (\bibinfo{year}{2023}).
\bibitem[{Tabandeh et~al.(2022)Tabandeh, Sharma, Iannacone, and
  Gardoni}]{Tabandeh22FPmixtures}
\bibinfo{author}{A.~Tabandeh}, \bibinfo{author}{N.~Sharma},
  \bibinfo{author}{L.~Iannacone}, \bibinfo{author}{P.~Gardoni},
\newblock \bibinfo{title}{Numerical solution of the {F}okker-{P}lanck equation
  using physics-based mixture models},
\newblock \bibinfo{journal}{Computer Methods in Applied Mechanics and
  Engineering} \bibinfo{volume}{399} (\bibinfo{year}{2022})
  \bibinfo{pages}{115424}.
  \DOIprefix\doi{https://doi.org/10.1016/j.cma.2022.115424}.
\bibitem[{Li and Meredith(2023)}]{Li2023time-dependentFP}
\bibinfo{author}{Y.~Li}, \bibinfo{author}{C.~Meredith},
\newblock \bibinfo{title}{Artificial neural network solver for time-dependent
  {F}okker-{P}lanck equations},
\newblock \bibinfo{journal}{Appl. Math. Comput.} \bibinfo{volume}{457}
  (\bibinfo{year}{2023}) \bibinfo{pages}{Paper No. 128185, 24}. \URLprefix
  \url{https://doi.org/10.1016/j.amc.2023.128185}.
  \DOIprefix\doi{10.1016/j.amc.2023.128185}.
\bibitem[{Lin et~al.(2022)Lin, Li, and Ren}]{Lin2022invariant-distribution}
\bibinfo{author}{B.~Lin}, \bibinfo{author}{Q.~Li}, \bibinfo{author}{W.~Ren},
\newblock \bibinfo{title}{Computing the invariant distribution of randomly
  perturbed dynamical systems using deep learning},
\newblock \bibinfo{journal}{J. Sci. Comput.} \bibinfo{volume}{91}
  (\bibinfo{year}{2022}) \bibinfo{pages}{Paper No. 77, 17}. \URLprefix
  \url{https://doi.org/10.1007/s10915-022-01844-5}.
  \DOIprefix\doi{10.1007/s10915-022-01844-5}.
\bibitem[{Lin et~al.(2023)Lin, Li, and Ren}]{Lin2023invariant-distribution}
\bibinfo{author}{B.~Lin}, \bibinfo{author}{Q.~Li}, \bibinfo{author}{W.~Ren},
\newblock \bibinfo{title}{Computing high-dimensional invariant distributions
  from noisy data},
\newblock \bibinfo{journal}{J. Comput. Phys.} \bibinfo{volume}{474}
  (\bibinfo{year}{2023}) \bibinfo{pages}{Paper No. 111783, 16}. \URLprefix
  \url{https://doi.org/10.1016/j.jcp.2022.111783}.
  \DOIprefix\doi{10.1016/j.jcp.2022.111783}.
\bibitem[{Mandal and Apte(2023)}]{mandal2023learning}
\bibinfo{author}{P.~Mandal}, \bibinfo{author}{A.~Apte},
\newblock \bibinfo{title}{Learning zeros of {F}okker-{P}lanck operators},
\newblock \bibinfo{journal}{arXiv: 2306.07068}  (\bibinfo{year}{2023}).
\bibitem[{Gu et~al.(2023)Gu, Harlim, Liang, and Yang}]{GuHLY23stationary}
\bibinfo{author}{Y.~Gu}, \bibinfo{author}{J.~Harlim},
  \bibinfo{author}{S.~Liang}, \bibinfo{author}{H.~Yang},
\newblock \bibinfo{title}{Stationary density estimation of {I}t\^{o} diffusions
  using deep learning},
\newblock \bibinfo{journal}{SIAM J. Numer. Anal.} \bibinfo{volume}{61}
  (\bibinfo{year}{2023}) \bibinfo{pages}{45--82}. \URLprefix
  \url{https://doi.org/10.1137/21M1445363}. \DOIprefix\doi{10.1137/21M1445363}.
\bibitem[{Lagaris et~al.(2000)Lagaris, Likas, and Papageorgiou}]{870037}
\bibinfo{author}{I.~Lagaris}, \bibinfo{author}{A.~Likas},
  \bibinfo{author}{D.~Papageorgiou},
\newblock \bibinfo{title}{Neural-network methods for boundary value problems
  with irregular boundaries},
\newblock \bibinfo{journal}{IEEE Transactions on Neural Networks}
  \bibinfo{volume}{11} (\bibinfo{year}{2000}) \bibinfo{pages}{1041--1049}.
  \DOIprefix\doi{10.1109/72.870037}.
\bibitem[{Kazem et~al.(2012)Kazem, Rad, and Parand}]{KAZEM2012181}
\bibinfo{author}{S.~Kazem}, \bibinfo{author}{J.~Rad},
  \bibinfo{author}{K.~Parand},
\newblock \bibinfo{title}{Radial basis functions methods for solving
  {F}okker-{P}lanck equation},
\newblock \bibinfo{journal}{Engineering Analysis with Boundary Elements}
  \bibinfo{volume}{36} (\bibinfo{year}{2012}) \bibinfo{pages}{181--189}.
\bibitem[{Dehghan and Mohammadi(2014)}]{DEHGHAN201438}
\bibinfo{author}{M.~Dehghan}, \bibinfo{author}{V.~Mohammadi},
\newblock \bibinfo{title}{The numerical solution of {F}okker-{P}lanck equation
  with radial basis functions ({RBF}s) based on the meshless technique of
  {K}ansa's approach and {G}alerkin method},
\newblock \bibinfo{journal}{Engineering Analysis with Boundary Elements}
  \bibinfo{volume}{47} (\bibinfo{year}{2014}) \bibinfo{pages}{38--63}.
\bibitem[{Wang et~al.(2023)Wang, Jiang, Hong, Zhao, and Sun}]{WANG2023103408}
\bibinfo{author}{X.~Wang}, \bibinfo{author}{J.~Jiang},
  \bibinfo{author}{L.~Hong}, \bibinfo{author}{A.~Zhao}, \bibinfo{author}{J.-Q.
  Sun},
\newblock \bibinfo{title}{Radial basis function neural networks solution for
  stationary probability density function of nonlinear stochastic systems},
\newblock \bibinfo{journal}{Probabilistic Engineering Mechanics}
  \bibinfo{volume}{71} (\bibinfo{year}{2023}) \bibinfo{pages}{103408}.
\bibitem[{Huang et~al.(2015)Huang, Ji, Liu, and
  Yi}]{HuangJLY15-existence-steadyFP}
\bibinfo{author}{W.~Huang}, \bibinfo{author}{M.~Ji}, \bibinfo{author}{Z.~Liu},
  \bibinfo{author}{Y.~Yi},
\newblock \bibinfo{title}{Steady states of {F}okker-{P}lanck equations: {I}.
  {E}xistence},
\newblock \bibinfo{journal}{J. Dynam. Differential Equations}
  \bibinfo{volume}{27} (\bibinfo{year}{2015}) \bibinfo{pages}{721--742}.
  \URLprefix \url{https://doi.org/10.1007/s10884-015-9454-x}.
  \DOIprefix\doi{10.1007/s10884-015-9454-x}.
\bibitem[{Levin and Peres(2017)}]{levin2017markov}
\bibinfo{author}{D.~A. Levin}, \bibinfo{author}{Y.~Peres},
  \bibinfo{title}{Markov chains and mixing times}, volume
  \bibinfo{volume}{107}, \bibinfo{publisher}{American Mathematical Soc.},
  \bibinfo{year}{2017}.
\bibitem[{Akian et~al.(2022)Akian, Bonnet, Owhadi, and
  Savin}]{Akian22learningbestkernel}
\bibinfo{author}{J.-L. Akian}, \bibinfo{author}{L.~Bonnet},
  \bibinfo{author}{H.~Owhadi}, \bibinfo{author}{E.~Savin},
\newblock \bibinfo{title}{Learning ``best" kernels from data in {G}aussian
  process regression. with application to aerodynamics},
\newblock \bibinfo{journal}{Journal of Computational Physics}
  \bibinfo{volume}{470} (\bibinfo{year}{2022}) \bibinfo{pages}{111595}.
  \DOIprefix\doi{10.1016/j.jcp.2022.111595}.
\bibitem[{Owhadi and Yoo(2019)}]{owhadi2019kernel}
\bibinfo{author}{H.~Owhadi}, \bibinfo{author}{G.~R. Yoo},
\newblock \bibinfo{title}{Kernel flows: From learning kernels from data into
  the abyss},
\newblock \bibinfo{journal}{Journal of Computational Physics}
  \bibinfo{volume}{389} (\bibinfo{year}{2019}) \bibinfo{pages}{22--47}.
\bibitem[{Yang et~al.(2024)Yang, Sun, Hamzi, Owhadi, and
  Xie}]{yang2024learning}
\bibinfo{author}{L.~Yang}, \bibinfo{author}{X.~Sun},
  \bibinfo{author}{B.~Hamzi}, \bibinfo{author}{H.~Owhadi},
  \bibinfo{author}{N.~Xie},
\newblock \bibinfo{title}{Learning dynamical systems from data: A simple
  cross-validation perspective, part v: Sparse kernel flows for 132 chaotic
  dynamical systems},
\newblock \bibinfo{journal}{Physica D: Nonlinear Phenomena}
  (\bibinfo{year}{2024}) \bibinfo{pages}{134070}.
\bibitem[{Nguyen et~al.(2022)Nguyen, Pham, Nguyen, Nguyen, Osher, and
  Ho}]{nguyen2022fourierformer}
\bibinfo{author}{T.~M. Nguyen}, \bibinfo{author}{M.~Pham},
  \bibinfo{author}{T.~M. Nguyen}, \bibinfo{author}{K.~Nguyen},
  \bibinfo{author}{S.~Osher}, \bibinfo{author}{N.~Ho},
\newblock \bibinfo{title}{{F}ourierformer: Transformer meets generalized
  {F}ourier integral theorem},
\newblock in: \bibinfo{editor}{A.~H. Oh}, \bibinfo{editor}{A.~Agarwal},
  \bibinfo{editor}{D.~Belgrave}, \bibinfo{editor}{K.~Cho} (Eds.),
  \bibinfo{booktitle}{Advances in Neural Information Processing Systems},
  \bibinfo{year}{2022}.
\bibitem[{Shin et~al.(2020)Shin, Zhang, and Karniadakis}]{shin2020error}
\bibinfo{author}{Y.~Shin}, \bibinfo{author}{Z.~Zhang}, \bibinfo{author}{G.~E.
  Karniadakis},
\newblock \bibinfo{title}{Error estimates of residual minimization using neural
  networks for linear pdes},
\newblock \bibinfo{journal}{arXiv:2010.08019}  (\bibinfo{year}{2020}).
\bibitem[{Mishra and Molinaro(2020)}]{mishra2020estimates}
\bibinfo{author}{S.~Mishra}, \bibinfo{author}{R.~Molinaro},
\newblock \bibinfo{title}{Estimates on the generalization error of physics
  informed neural networks (pinns) for approximating pdes},
\newblock \bibinfo{journal}{arXiv:2006.16144}  (\bibinfo{year}{2020}).
\bibitem[{Bogachev et~al.(2022)Bogachev, Krylov, R{\"o}ckner, and
  Shaposhnikov}]{bogachev2022fokker}
\bibinfo{author}{V.~I. Bogachev}, \bibinfo{author}{N.~V. Krylov},
  \bibinfo{author}{M.~R{\"o}ckner}, \bibinfo{author}{S.~V. Shaposhnikov},
  \bibinfo{title}{Fokker--Planck--Kolmogorov Equations}, volume
  \bibinfo{volume}{207}, \bibinfo{publisher}{American Mathematical Society},
  \bibinfo{year}{2022}.

\end{thebibliography}
\appendix

\section{Universal approximation of tensor neural networks} \label{app:universal}

We use a similar argument in \cite{nguyen2022fourierformer} to show the universal approximation property of tensor neural networks. 
Recall from the Fourier integral theorem, for any $p(x)\in L^1(\Real^d)$, 
\begin{equation*}
    p(x) = \frac{1}{(2\pi)^d} \int_{\Real^d} \int_{\Real^d} \cos(z^\top (x-y)) p(y)\,dz\,dy
\end{equation*}
Observe that 
\begin{equation*}
\int_{\Real^d} \cos(z^\top (x-y))  \,dz
=\lim_{R\to \infty }   \int_{[-R,R]^d} \cos(z^\top (x-y)) \,dz 
     = \lim_{R\to \infty }\bigotimes_{j = 1}^{d} \frac{\sin(R(x_{j} - y_{j}))}{(x_{j} - y_{j})}. 
\end{equation*}
We then approximate  the integral 
\begin{equation*}
    p(x) = \lim_{R\to \infty }\frac{1}{(2\pi)^d} \int_{[-R,R]^d} \cK_R(x-y)p(y) \,dy,\quad \cK_R(x-y) = \bigotimes_{j = 1}^{d} \frac{\sin(R(x_{j} - y_{j}))}{(x_{j} - y_{j})} = \prod_{j = 1}^{d}  K_{R,j}(x_j-y_j).
    \end{equation*} 
Specifically, we have the error $\varepsilon_R $ from approximating the whole space with $[-R, R]^d$ and the error $\varepsilon_{\rm int}$ of numerical integration over this bounded domain. 
\begin{equation*}
    p(x) = \int_{[-R,R]^d} \cK_R(x-y)p(y) \,dy +\varepsilon_R  =
 \sum_{\abs{\alpha}=1}^N c_\alpha \prod_{j = 1}^{d}  K_{R,j}(x_{j}-y_{j,\alpha_i})+ \varepsilon_{\text{int}}  +\varepsilon_R.
\end{equation*}
Here $y_{j,\alpha_i}$ are sampling points  along the direction $y_j$ and $c_\alpha$ depends on $p(y)$ and $y_{j,\alpha_i}$.
The numerical integration can be realized by  the Monte Carlo method
when   $p(y)$ is continuous on $[-R, R]^d$ and has second-order moments. 
This shows a universal approximation of the tensor-product kernels. Observe that the kernel in each dimension is continuous and thus can be approximated by another kernel function or neural network with universal approximation property in each dimension. 

\section{Stability of the least-square formulation}
According to the theory in \cite{shin2020error,mishra2020estimates}, it requires the stability of the exact solution for the error estimates of the neural network solutions for physics-informed neural networks (least-square collocation formulations). In the following, we will focus on the stability of the solution. 

Let's consider 
the following auxiliary problem.
   \begin{equation}\label{eq:fokker-planck-steady-rhs}
\mathcal{L}q= - \sum_{i}\frac{\partial }{\partial x_i}\left( f_i q \right)(x) + \frac{1}{2} \sum_{i, j} \frac{\partial^2 }{\partial x_i \partial x_j} \left( D_{ij} q \right)(x)=g(x),\quad x \in  \Omega,
\end{equation}
with $\displaystyle \int_{\Omega} q(x) dx = 1, q\geq 0$ and $q(x)=0$ on $\partial\Omega$.
When the solution to \eqref{eq:fokker-planck-steady} $p$ vanishes outside $\Omega$, it satisfies the above equation when $g=0$. 

We assume that there exists a unique solution to this problem. For discussions about the uniqueness, we refer to the book
\cite{bogachev2022fokker}. 
Moreover, we assume that the matrix $D$ is positive definite and 
\begin{equation}\label{eq:uniform-ellipticity}
 \sum_{i,j}D_{i,j}\xi_i\xi_j\geq C_{{D}}\abs{\bm \xi}^2, \text{ for all } \bm\xi \in \Real^d.   
\end{equation}
With additional assumptions on $\mathbf{f}$, we can obtain the stability of the solution. Let $\mathbf{b}(x) \in \Real^d$, where $b_i= \frac{1}{2}\sum_{j=1}^d\partial_{x_j}D_{i,j} -f_i$. 
Let $C_P$ depending on the volume of $\Omega$ be the constant from the Poincare inequality: 
\begin{equation*}
  \int_\Omega q^2\,dx\leq C_P \int_{\Omega} (\nabla q)^2\,dx,\quad q(x)=0 \text{ for } x \in \partial\Omega.
\end{equation*} 

\begin{thm}\label{thm:stabilty-fpe}
Let $C_P$ be the constant from the Poincare inequality.   Assume that \begin{equation}
        \frac{C_D}{C_P}-\nabla \cdot \mathbf{b}\geq C_0>0.
    \end{equation}
    where  $C_D$ is from \eqref{eq:uniform-ellipticity}, $\mathbf{b}$ is defined above and $C_0$ is a positive constant. 
    Then we have 
    \begin{equation}\label{eq:stability-fpe}
        C_0^2\int_{\Omega}q^2\,dx \leq 
        4 \int_{\Omega} (\mathcal{L}q)^2\,dx.  
    \end{equation}
  
\end{thm}

\begin{proof}
Multiplying $-q$ and integrating over both sides of Equation \eqref{eq:fokker-planck-steady-rhs}, we have  $
   - \int_\Omega\mathcal{L}q q = \int_{\Omega} -q(x) g(x).$
Applying integration-by-parts several times gives 
\begin{align}\label{eq:int-by-parts}
   \frac{1}{2}\int_{\Omega}\sum_{i,j} 
D_{i,j}\partial_{x_i}q \partial_{x_j}q -  (\nabla \cdot \mathbf{b})  q^2\,dx = \int_{\Omega} -q(x) g(x).
\end{align}
where $b_i= \frac{1}{2}\sum_{j=1}^d\partial_{x_j}D_{i,j} -f_i$.  
Applying the uniform ellipticity 
\eqref{eq:uniform-ellipticity} and   the Poincare  inequality, we have   
\begin{align*}
    \int_\Omega q^2 (\frac{C_D}{C_P}-\nabla \cdot \mathbf{b})\,dx &\leq 
     \int_\Omega C(\nabla q)^2 -q^2\nabla \cdot \mathbf{b})\,dx\\
     &\leq \int_{\Omega}\sum_{i,j} 
D_{i,j}\partial_{x_i}q \partial_{x_j}q -  (\nabla \cdot \mathbf{b})  q^2\,dx.
    \end{align*}
    Then by Equation \eqref{eq:fokker-planck-steady-rhs} and \eqref{eq:int-by-parts}, we have 
    \begin{align*}
       \int_\Omega q^2 (\frac{C}{C_P}-\nabla \cdot \mathbf{b})\,dx &\leq  -2\int_{\Omega} g(x)q\,dx =   -2\int_{\Omega}\sum_{i,j} 
\mathcal{L}q q\,dx.
    \end{align*}
    Then by the Cauchy-Schwarz inequality, we obtain the desired stability \eqref{eq:stability-fpe}. 
 %
\end{proof}

    %

     By  the analysis of least-square collocation methods using networks \cite{shin2020error,mishra2020estimates} 
     and Theorem \ref{thm:stabilty-fpe}, we obtain   
     approximation errors from the proposed methodology with loss \eqref{eqn: Loss}.  
     Specifically,     when the solution $p$ to Equation \eqref{eq:fokker-planck-steady} vanishes outside  $\Omega$, we may use the linearity of the operator and 
     Theorem \ref{thm:stabilty-fpe}, 
     \begin{equation*}
      C_0^2\int_{\Omega}(p-p_{N})^2\,dx \leq 
          4\int_{\Omega} (\mathcal{L}p-p_{N})^2\,dx.
        \end{equation*}
      For the last term in the above inequality, we use 
$M_T$ Monte Carlo  points
of uniform distribution to approximate the integral over $\Omega$. Suppose that  the Monte Carlo integration error for 
     $\displaystyle\int_{\Omega} (\mathcal{L}p_{N})^2\,dx$ is the   error $\varepsilon_T $.
Then we obtain the following generalization error estimate for  network solutions:
          \begin{equation*}
      C_0^2\int_{\Omega}(p-p_{N})^2\,dx \leq 
          4\int_{\Omega} (\mathcal{L}p-p_{N})^2\,dx
          = 4  \frac{1}{\# \,\text{samples}}\sum_{i=1}^{\# \,\text{samples}}\Big({\mathcal{L}\displaystyle p_N(x_i;\Theta)\Big)}^2
          +4\varepsilon_{\# \,\text{samples}}.
        \end{equation*} 
    The first term in the last inequality is proportional to the loss function \eqref{eqn: Loss} (after training) and $\varepsilon_{\# \,\text{samples}}$ is the Monte Carlo integration error- the difference between the integral and the finite sum. 
     We refer to the detailed analysis in \cite{shin2020error}, where more general linear equations than the Fokker-Planck equations are considered.  

\section{Hyper-parameters in Section \ref{sec:bounded domain} for numerical examples  }\label{sec:append-sde}
Here we report all the hyperparameters used in our examples.

\begin{table}[ht]
\centering
\begin{tabular}{|c|c|}
\hline
\textbf{Parameter}    & \textbf{Values} \\ \hline
Step size ($h$)              & 0.001           \\ \hline
Burn-in time ($t_{burnin}$)  & $1 \times 10^6$     \\ \hline
Terminal time ($t_{terminal}$) & $1.5 \times 10^6$   \\ \hline
Number of trajectories($Q$)     & 10            \\ \hline 
$C$    & 1.1          \\ \hline 
\end{tabular}
\caption{Hyper-parameter values in Algorithm \ref{alg:obtain-bounded-domain} for all test examples.}
\end{table}

\subsection{Hyperparameters in Algorithm 
\ref{alg:refine-support}}\label{sec:append-partition}
Here we report all the hyperparameters we used in  the Algorithm 
\ref{alg:refine-support}. 
In the second rows of Tables \ref{tbl:exm42-refining} -\ref{tbl:exm44-refining}, the integration means $
\int_{\bigotimes_{j=1}^d [O_j - r, O_j + r]} p_N(x) \, dx
$, where $p_N$ is the trained neural network and $r$ is the element in $S_n$.

\begin{table}[ht]
\centering
\begin{tabular}{c|c|c|c|c|c|c|c}
\hline
 $r \in S_n$ &0.4&0.8&1.2&1.6&2.0&2.4& $B= 2.6472$  \\ 
\hline
Integration &0.138 & 0.703  & 0.964 & 0.996& 0.999 & 0.999 & 1.0  \\
\hline
\end{tabular}
\caption{Example \ref{exm:4.2-4dunimode}: elements in $S_n$ and their integration results with the trained TRBFN.}
\label{tbl:exm42-refining}
\end{table}
\begin{table}[!htp]
\centering
\begin{tabular}{c|c|c|c|c|c|c|c|c}
\hline
 $r \in S_n$ &0.2&0.4&0.6&0.8&1.0&1.2&1.4&$B=1.5191$ \\ 
\hline
Integration &0.002 & 0.081  & 0.344 & 0.739 & 0.979 & 0.998 & 0.999 & 1.0  \\
\hline
\end{tabular}
\caption{Example \ref{exm:6d-unimode}: elements in $S_n$ and their integration results with the trained TRBFN.}\label{tbl:exm43-refining}
\end{table}
\vskip -10pt
\begin{table}[!ht]
\centering
\scalebox{0.8}{
\begin{tabular}{c|c|c|c|c|c|c|c|c|c|c|c|c}
\hline
 $r \in S_n$ &1.2&1.6&2.0&2.4&2.8&3.2&3.6& 4.0&4.4&4.8&5.2& $B=5.2872$\\ 
\hline
Integration & 0.227 & 0.494 & 0.661& 0.752 & 0.809 & 0.853 & 0.896 & 0.940 & 0.978 & 0.997 & 0.999 & 1.0  \\
\hline
\end{tabular}
}
\caption{Example \ref{exm:6d-multi-mode}: elements in $S_n$ and their integration results with the trained TRBFN.}
\label{tbl:exm44-refining}
\end{table}


\subsection{Weights for TRBFN}\label{sec:append-Weights}
We report all weights for penalty terms in the loss function \eqref{eqn:trbfn-loss}.
In  Examples \ref{exm:2d-ring}- \ref{exm:6d-multi-mode}, we use $\mathcal{W}_1=
 50000$      and $\mathcal{W}_2= 100$.  
In  Examples \ref{exm:6d-multi-mode}, we use $\mathcal{W}_1=\mathcal{W}_2= 100$.

\section{Training Hyper-parameters}\label{sec:append-training}
We list the batch sizes for all examples in Section \ref{sec:numerical}, for both TFNNs and TRBFNs, in Table \ref{tbl:hyperparamers}. 
\begin{table}[h!]
\centering
\begin{tabular}{lllllllll}
\toprule
\textbf{Example} & \textbf{Method} & \textbf{Batch Size} \\
\midrule
Example \ref{exm:2d-ring} & TFFN & 4096 \\
Example \ref{exm:2d-ring} & TRBFN & 5000  \\
Example \ref{exm:4.2-4dunimode} & TFFN & 4096 \\
Example \ref{exm:4.2-4dunimode} & TRBFN & 5000   \\
Example \ref{exm:6d-unimode} & TFFN & 16000  \\
Example \ref{exm:6d-unimode} & TRBFN & 5000  \\
Example \ref{exm:6d-multi-mode} & TRBFN & 40000 \\
Example \ref{exm:10d-multi-mode} & TRBFN & 10000 \\
\bottomrule
\end{tabular}
\caption{\tw{Training batch sizes for TRBFNs and TFNNs for different examples in 
Section \ref{sec:numerical}.}}
\label{tbl:hyperparamers}
\end{table}


\end{document}